\documentclass[12pt, reqno]{amsart}
\usepackage{amssymb, mathrsfs}
\usepackage{latexsym}
\usepackage{amsmath}
\usepackage{amsthm}
\usepackage{amsfonts}
\usepackage{dsfont, upgreek}
\usepackage{mathtools}
\usepackage{epsfig}
\usepackage{amscd}
\usepackage{graphicx}
\usepackage{tikz-cd}
\usepackage[shortlabels]{enumitem}
\setlist[enumerate]{itemsep=0.5ex}

\usepackage{color}
\usepackage{tikz}
\usetikzlibrary{patterns}

\usepackage[left=1.4in,right=1.4in,top=1.2in,bottom=1.2in]{geometry}

\usepackage{tikz}
\usetikzlibrary{decorations.markings}
\usetikzlibrary{arrows.meta}

\usepackage{extarrows}

\usepackage{adjustbox}
\usepackage{pgfplots}
\usepgfplotslibrary{colormaps}
\usepackage{slashed}

\usepackage[colorlinks=true,linkcolor=blue,citecolor=blue]{hyperref}

\usepackage[colorinlistoftodos,prependcaption,textsize=tiny]{todonotes}  

\usepackage[all,cmtip]{xy}
\theoremstyle{plain}
\newtheorem{theorem}{Theorem}[section]
\newtheorem{proposition}[theorem]{Proposition}
\newtheorem{lemma}[theorem]{Lemma}

\theoremstyle{definition} 
\newtheorem{definition}[theorem]{Definition}

\newtheorem*{claim*}{Claim}
\newtheorem{setup}[theorem]{Geometric Setup}

\theoremstyle{remark} 
\newtheorem{remark}[theorem]{Remark}

\numberwithin{equation}{section}

\newcommand{\Sc}{\mathrm{Sc}}

\newcommand{\Endo}{\mathrm{End}}
\newcommand{\Bigwedge}{\mathord{\adjustbox{raise=.4ex, totalheight=.7\baselineskip}{$\bigwedge$}}}

\newcommand{\tcon}{\widehat \nabla}
\newcommand{\tdirac}{\widehat D}
\newcommand{\grad}{\overbar d}
\DeclareMathOperator{\hess}{\mathrm{Hess}}

\newcommand{\ind}{\textup{Ind}}
\newcommand{\id}{\mathrm{id}}

\newcommand{\R}{\mathbb{R}}
\newcommand{\tr}{\mathrm{tr}}

\newcommand{\Z}{\mathbb{Z}}

\newcommand{\mcon}{\prescript{M}{}{\nabla}}
\newcommand{\ncon}{\prescript{N}{}{\nabla}}

\newcommand{\interior}[1]{%
	{\kern0pt#1}^{\mathrm{\,o}}%
}

\makeatletter
\let\save@mathaccent\mathaccent
\newcommand*\if@single[3]{%
	\setbox0\hbox{${\mathaccent"0362{#1}}^H$}%
	\setbox2\hbox{${\mathaccent"0362{\kern0pt#1}}^H$}%
	\ifdim\ht0=\ht2 #3\else #2\fi
}
\newcommand*\rel@kern[1]{\kern#1\dimexpr\macc@kerna}
\newcommand*\overbar[1]{\@ifnextchar^{{\wide@bar{#1}{0}}}{\wide@bar{#1}{1}}}
\newcommand*\wide@bar[2]{\if@single{#1}{\wide@bar@{#1}{#2}{1}}{\wide@bar@{#1}{#2}{2}}}
\newcommand*\wide@bar@[3]{%
	\begingroup
	\def\mathaccent##1##2{%
		\let\mathaccent\save@mathaccent
		\if#32 \let\macc@nucleus\first@char \fi
		\setbox\z@\hbox{$\macc@style{\macc@nucleus}_{}$}%
		\setbox\tw@\hbox{$\macc@style{\macc@nucleus}{}_{}$}%
		\dimen@\wd\tw@
		\advance\dimen@-\wd\z@
		\divide\dimen@ 3
		\@tempdima\wd\tw@
		\advance\@tempdima-\scriptspace
		\divide\@tempdima 10
		\advance\dimen@-\@tempdima
		\ifdim\dimen@>\z@ \dimen@0pt\fi
		\rel@kern{0.6}\kern-\dimen@
		\if#31
		\overline{\rel@kern{-0.6}\kern\dimen@\macc@nucleus\rel@kern{0.4}\kern\dimen@}%
		\advance\dimen@0.4\dimexpr\macc@kerna
		\let\final@kern#2%
		\ifdim\dimen@<\z@ \let\final@kern1\fi
		\if\final@kern1 \kern-\dimen@\fi
		\else
		\overline{\rel@kern{-0.6}\kern\dimen@#1}%
		\fi
	}%
	\macc@depth\@ne
	\let\math@bgroup\@empty \let\math@egroup\macc@set@skewchar
	\mathsurround\z@ \frozen@everymath{\mathgroup\macc@group\relax}%
	\macc@set@skewchar\relax
	\let\mathaccentV\macc@nested@a
	\if#31
	\macc@nested@a\relax111{#1}%
	\else
	\def\gobble@till@marker##1\endmarker{}%
	\futurelet\first@char\gobble@till@marker#1\endmarker
	\ifcat\noexpand\first@char A\else
	\def\first@char{}%
	\fi
	\macc@nested@a\relax111{\first@char}%
	\fi
	\endgroup
}
\makeatother

\begin{document}

\title{Dihedral rigidity for submanifolds of warped product manifolds}

\author{Jinmin Wang}
\address[Jinmin Wang]{Department of Mathematics, Texas A\&M University}
\email{jinmin@tamu.edu}
\thanks{The first author is partially supported by NSF 1952693.}
\author{Zhizhang Xie}
\address[Zhizhang Xie]{ Department of Mathematics, Texas A\&M University }
\email{xie@math.tamu.edu}
\thanks{The second author is partially supported by NSF 1952693.}

\begin{abstract}
	In this paper, we prove a dihedral extremality and rigidity theorem for a large class of codimension zero submanifolds with polyhedral boundary  in   warped product manifolds. We remark that the spaces considered in this paper are not necessarily warped product manifolds themselves. In particular, the results of this paper are applicable to  submanifolds (of warped product manifolds)  with faces that are neither orthogonal nor parallel to the radial direction of the warped product metric. Generally speaking, the dihedral rigidity results require  the leaf of the underlying warped space to have positive Ricci curvature and the warping function to be strictly log-concave. Nevertheless, we prove a dihedral rigidity theorem  for  a large class of hyperbolic polyhedra, where the leaf of the underlying warped product space is flat and the warping function is not strictly log-concave. 
\end{abstract}
\maketitle
	\section{Introduction}
The main purpose of this paper is to prove Gromov's  extremality and dihedral rigidity conjectures
  for a large class of submanifolds with polyhedral boundary in   warped product spaces. The main results of the paper are generalizations of results in   \cite[Section 5.6]{Gromov4lectures2019}, \cite{MR4158684} and \cite{Wang:2021tq} to a wider class of manifolds   with polyhedral boundary. See Definition \ref{def:polytopeboundary} below for the precise definition of manifolds with polyhedral boundary.

The Gromov dihedral extremality and rigidity conjectures concern comparisons of scalar curvature, mean curvature and dihedral angles for Riemannian metrics on manifolds with polyhedral boundary. They can be viewed as  scalar curvature analogues of the Alexandrov’s triangle comparisons for spaces whose sectional curvature is bounded below \cite{MR0049584,MR854238}. The dihedral extremality and  rigidity conjectures are two important conjectures  among a list of conjectures and open questions on  scalar curvature formulated by Gromov  \cite{GromovDiracandPlateau, Gromovinequalities2018, 	Gromov4lectures2019}. They have profound implications in  geometry and mathematical physics. For example, it implies the positive mass theorem, a foundational result in general relativity and differential geometry
\cite{MR535700,MR612249, MR626707} (cf. \cite[Section 5]{Li:2019tw} and \cite[Discussion after Theorem 1.7] {Wang:2021tq}).  

In a previous paper joint with Yu \cite{Wang:2021tq}, the authors settled   Gromov's dihedral extremality conjecture for convex Euclidean polyhedra in all dimensions \cite[Theorem 1.8]{Wang:2021tq}. Later on,  the authors answered positively Gromov's flat corner domination conjecture  in all dimensions  \cite{Wang2022:fl}. As a consequence,  the authors answered positively the Stoker conjecture on convex Euclidean polyhedra in all dimensions \cite{Wang2022:fl}. We refer the reader to \cite{Wang2022:fl, Wang:2021tq} and the references therein for more details.

While the results and methods in \cite{Wang2022:fl, Wang:2021tq} concern manifolds with nonnegative curvature operator, we prove in this paper Gromov's  extremality and dihedral rigidity conjectures
for a large class of manifolds with polyhedral boundary that are not necessarily non-negatively curved (e.g. hyperbolic manifolds with polyhedral boundary).  Let us first fix some notation. Let $(X,h)$ be a smooth $(n-1)$ dimensional Riemannian manifold with polyhedral boundary and $\varphi$ a positive function. A warped product metric on $X\times I$ is a Riemannian metric of the following form
$$g=dr^2+\varphi(r)^2h,$$
where $r$ is the parameter of the interval $I$. We denote by $\partial_r$ the tangent vector along the $r$ direction.  We should mention that some special cases of the results of the present paper   were previously proved by Gromov (for the case where  the leaf $X$ is a regular spherical simplex \cite[Section 5.6]{Gromov4lectures2019})  and by  Cecchini-Zeidler  (for the case where the leaf $X$ is an even dimensional closed manifold \cite[Section 10]{Cecchini:2021vs}). However, while the previous results of  Gromov \cite[Section 5.6]{Gromov4lectures2019} and  Cecchini-Zeidler \cite[Section 10]{Cecchini:2021vs} deal with actual warped product spaces, we emphasize that the manifolds we study in the present paper are \emph{not} necessarily warped product manifolds themselves, but  \emph{submanifolds} of warped product manifolds. To be precise, we work for the following regions.
\begin{definition}\label{def:radiallyConvex}
	Let $M$ be a codimension zero submanifold with polyhedral boundary in $(X\times I,dr^2+\varphi(r)^2 h)$. We say that $M$ is \emph{radially convex} if on each face of $M$, its second fundamental form $A$ satisfies
	$$A-\frac{\varphi'}{\varphi}\langle\nu,\partial r\rangle\geq 0,$$
	where $\nu$ is the unit outer normal vector of the face, and $\partial r$ is the gradient of the function $r$.
\end{definition} 

See Figure \ref{fig:example} for some examples of radially convex domains in a Euclidean annulus $S^1\times I$. 
\begin{figure}[h]
	\begin{tikzpicture}[scale=1]
		\draw[thick] (4,0) arc(0:90:4) (0,4) -- (1.414,1.414) -- (4,0);
		\filldraw[opacity=0.1] (4,0) arc(0:90:4) (0,4) -- (1.414,1.414) -- (4,0);
		\draw[dashed] (1.6,0) -- (4,0);
		\draw[dashed] (0,1.6) -- (0,4);
		\draw[dashed] (1.6,0) arc(0:90:1.6);
	\end{tikzpicture}
	\begin{tikzpicture}
		\draw[thick] (4,0) arc(0:90:4) (0,4) -- ({2*cos(60)},{2*sin(60)}) arc(60:30:2) ({2*sin(60)},{2*cos(60)}) -- (4,0);
		\filldraw[opacity=0.1] (4,0) arc(0:90:4)  -- ({2*cos(60)},{2*sin(60)}) arc(60:30:2) ({2*sin(60)},{2*cos(60)}) -- (4,0);
		\draw[dashed](4,0) -- (2,0) arc(0:90:2) -- (0,4);
	\end{tikzpicture}
	\begin{tikzpicture}
		\draw[thick] (2,0) arc(0:270:2) (0,-2) -- (0,-1) arc(270:0:1) (1,0) -- (2,0) ;
		\filldraw[opacity=0.1] (2,0) arc(0:270:2)  -- (0,-1) arc(270:0:1) (1,0) -- (2,0) ;
	\end{tikzpicture}
	\caption{Examples of radially convex domains in $\R^2$ with $X=S^1$ and $\varphi(r)=r^2$}
	\label{fig:example}
\end{figure}
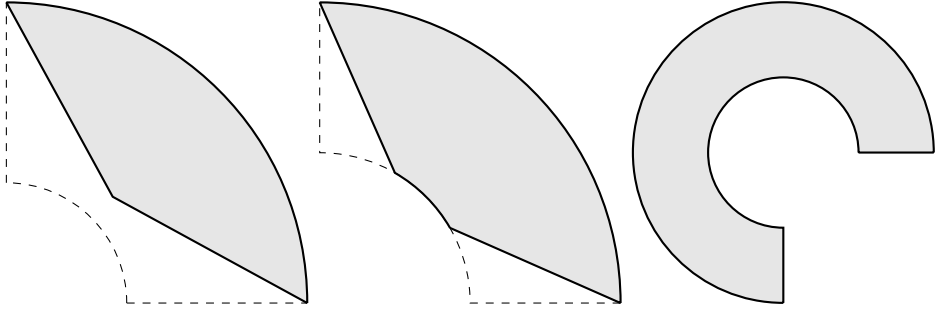

Note that some of the faces of $M$ could have negative curvature. 

\begin{definition}
A  map $f\colon N \to M$ between two oriented  manifolds $N$ and $M$ is called a spin map if the second Stiefel--Whitney classes of $TM$ and $TN$ are related by 
		\[  w_2(TN) = f^\ast(w_2(TM)).\]
		Equivalently, $f\colon N \to M$ is a spin map   if $TN\oplus f^\ast TM$ admits a spin structure. 
\end{definition}

Given a Riemannian metric $g$ on an oriented manifold $M$ with polyhedral boundary (cf. Definition \ref{def:polytopeboundary}), we shall denote the scalar curvature of $g$ by $\Sc(g)$, the mean curvature\footnote{Our sign convention for the mean curvature is that the mean curvature of the standard round sphere viewed as the boundary of a Euclidean ball is positive. } of each face $F_i$ of $M$ by $H_g(F_i)$, and the dihedral angle function of two adjacent faces $F_i$ and $F_j$ by $\theta_{ij}(g)$.  The first main result of this paper is the following.

\begin{theorem}\label{thm:rigidityOfWarpedProductintro}
Suppose $M$ is an $n$-dimensional radially convex domain  in a warped product manifold $(X\times I, g)$, where  $g=dr^2+\varphi^2h$, such that 
	\begin{enumerate}[label=$(\alph*)$]
		
		\item the curvature operator for $(X, h)$ is non-negative,
		\item $\varphi$ is increasing (i.e. $\varphi'\geq 0$)  and log-concave, that is,  
		\[ (\log\varphi)'' = \left(\frac{\varphi'}{\varphi}\right)'  \leq 0, \]
			  \item the Euler characteristic of $M$ is  non-zero.
	\end{enumerate}
If $(N, \overbar g) $ is a Riemannian manifold with polyhedral boundary and $f\colon N\to M$  is a  spin  polytope\footnote{See Definition \ref{def:polytopeMap} below for the precise definition of polytope maps.} map such that 
	\begin{enumerate}[label=$(\arabic*)$]
		
		\item 
		$\Sc(\overbar{g})_x \geq \Sc(g)_{f(x)}$ for all $x\in N$, 
		\item $H_{\overbar{g}}(\overbar{F}_i)_y \geq  H_{g}(F_i)_{f(y)}$ for all $y$ in each codimension one face $\overbar{F}_i$ of $N$, 
		\item $\theta_{ij}(\overbar{g})_z\leq \theta_{ij}(g)_{f(z)}$ for  all $z \in \overbar F_i\cap \overbar{F}_j$, where $\overbar F_i$ and $ \overbar F_j$ are adjacent codimension one faces of $N$, 
		\item $f\colon N\to M$ is distance non-increasing,\footnote{Here $f\colon N\to M$ is said to be distance non-increasing at $x\in N$ if $\|df\|_x\leq 1$, where $df\colon  TN\to TM$ is the tangent map. We say $f$ is distance non-increasing if $f$ is distance non-increasing   at all $x\in N$. } 
			\item the degree of $f$ is nonzero,
	\end{enumerate} 
	then we have 
	\begin{enumerate}[label=$(\roman*)$]
		
		\item 
		$\Sc(\overbar{g})_x = \Sc(g)_{f(x)}$ for all $x\in N$, 
		\item $H_{\overbar{g}}(\overbar{F}_i)_y =  H_{g}(F_i)_{f(y)}$ for all $y$ in each codimension one face $\overbar{F}_i$ of $N$, 
		\item $\theta_{ij}(\overbar{g})_z =  \theta_{ij}(g)_{f(z)}$ for  all $z \in \overbar F_i\cap \overbar{F}_j$,  
		\item and if more than two codimension one faces of $N$ meet at a point, then the angle\footnote{The angle between a pair of non-adjacent codimension one faces of $N$ is \emph{not} a dihedral angle. In particular, the conclusion of part $(iv)$ of Theorem \ref{thm:rigidityOfWarpedProductintro} states that not only the corresponding dihedral angles of $N$ and $M$ are equal, but also the other  corresponding angles of $N$ and $M$ are equal.}  between any pair of inner normal vectors of these faces of $N$ is equal to the angle between the inner normal vectors of the corresponding faces of  $M$.
	\end{enumerate}   Moreover, the following hold. 
	\begin{enumerate}[label=$(\mathrm{\Roman*})$]
		\item  If $\varphi$ is strictly log-concave,\footnote{$\varphi$ is strictly log-concave if $(\log\varphi)''<0.$} then the metric $\overbar g$ on $N$ is also a warped product metric of the form $d\overbar r^2+\overbar \varphi^2\overbar h$ with $f_*(\partial_{\overbar r})=\partial_r$ and $\overbar\varphi=f^*\varphi$. If additionally the metric $h$ on $X$ is Ricci positive, then $f$ is a local isometry.
		\item For  $x$  in a codimension one face $\overbar F$ of $N$, if
		$$A-\frac{\varphi'}{\varphi}\langle\nu,\nabla r\rangle>0,$$
		at $x$, then $\partial f$ is an isometry\footnote{Here $\partial f \coloneqq  f|_{\partial N}$ is the restriction of $f\colon N\to M$ on  $\partial N$. }  at $x$.
	\end{enumerate}
\end{theorem}

Theorem \ref{thm:rigidityOfWarpedProductintro} is a generalization of the main results in a previous paper of Yu and the authors \cite{Wang:2021tq} to a  wider class of manifolds. More precisely, while  the main results in  \cite{Wang:2021tq} require the  curvature operator of the target space $(M, g)$ to be non-negative, Theorem \ref{thm:rigidityOfWarpedProductintro} only requires the curvature operator of the leaf $(X, h)$ to be non-negative. In particular, Theorem \ref{thm:rigidityOfWarpedProductintro} applies to manifolds that are not necessarily non-negatively curved, for examples, to  certain hyperbolic polyhedra (cf. Theorem \ref{thm:hyperbolicComparison-intro}).  There are two main ingredients for our proof of Theorem \ref{thm:rigidityOfWarpedProductintro}: the first  is to apply the index theory for manifolds with polyhedral boundary developed in \cite{Wang:2021tq}; and the second  is to modify the Dirac operator by a potential that is determined by the warped product metric $g = dr^2 + \varphi^2 h$ on $X\times I$ (see  Equation \eqref{eq:Diracpotential} for the details). The specific choice of such a potential is inspired by the work of Cecchini and Zeidler on scalar-mean rigidity results on warped products of closed manifolds \cite{Cecchini:2021vs}, which in turn was inspired by Gromov's $\mu$-bubble approach to the extremality and rigidity phenomena  of certain warped product spaces \cite[Section 5.6]{Gromov4lectures2019}. We emphasize that the manifolds we study in the present paper are \emph{not} necessarily warped product manifolds themselves, but  \emph{submanifolds} of warped product manifolds. In particular, these submanifolds can have faces that are neither orthogonal nor parallel to the radial direction.\footnote{Here the radial direction of a warped product metric $g = dr^2 + \varphi^2 h$ is the $r$-direction}   As far as the dihedral extremality is concerned,    Theorem \ref{thm:rigidityOfWarpedProductintro} only requires the warping function $\varphi$ to be log-concave instead of strictly log-concave. Generally speaking, the dihedral rigidity part of Theorem \ref{thm:rigidityOfWarpedProductintro} \emph{does} require the stronger assumption that 
$\varphi$ is strictly log-concave. However,  by adapting the techniques developed by the authors in \cite{Wang2022:fl}, we shall prove a dihedral rigidity theorem in some cases even when the warping function is not strictly log-concave. We shall explain the details in  Theorem \ref{thm:hyperbolicComparison-intro} and Theorem \ref{thm:rigidityOfProduct-intro} below.

Let us first consider the case where the warping function $\varphi$ satisfies  $(\log\varphi)''=0$, which is log-concave but clearly not strictly log-concave. A simple computation shows that   $\varphi = a e^{b r}$ for some constants $a$ and $b$ in this case. Note that, if $(X,h)$ is $\R^{n-1}$ equipped with the flat metric, then  the metric $g=dr^2+e^{2r}h$ is exactly the hyperbolic metric on $X\times I$. Our second main theorem of the paper states that the dihedral  rigidity theorem still holds in this case. More precisely, we have the following theorem. 

\begin{theorem}\label{thm:hyperbolicComparison-intro}
	Suppose $M$ is a compact  radially convex domain in the hyperbolic space $\mathbb H^n$ in the sense of Definition \ref{def:radiallyConvex},  where the hyperbolic metric $g$ on $\mathbb H^n$ is viewed as the following warped product 
	$$g=dr^2+e^{2r}h.$$ Let $N$ be an $n$-dimensional manifold with polyhedral boundary  and  $f\colon (N, \overbar g)\to (M, g)$ a spin polytope map such that 
	\begin{enumerate}[label=$(\arabic*)$]
		\item 
		$\Sc(\overbar{g})_x \geq \Sc(g)_{f(x)}=-(n-1) $ for all $x\in N$, 
		\item $H_{\overbar{g}}(\overbar{F}_i)_y \geq  H_{g}(F_i)_{f(y)}$ for all $y$ in each codimension one face $\overbar{F}_i$ of $N$, 
		\item $\theta_{ij}(\overbar{g})_z\leq \theta_{ij}(g)_{f(z)}$ for all $\overbar F_i, \overbar F_j$ and all $z \in \overbar F_i\cap \overbar{F}_j$,   
		\item the degree of $f$ is nonzero and  $f$ is distance non-increasing on the boundary, 
		\item  the Euler characteristic of $M$ is nonzero, 
	\end{enumerate} 
	then we have the following.
	\begin{enumerate}[label=$(\roman*)$]	
		\item 
		$\Sc(\overbar{g})_x = \Sc(g)_{f(x)}$ for all $x\in N$, 
		\item $H_{\overbar{g}}(\overbar{F}_i)_y =  H_{g}(F_i)_{f(y)}$ for all $y$ in each codimension one face $\overbar{F}_i$ of $N$, 
		\item $\theta_{ij}(\overbar{g})_z =  \theta_{ij}(g)_{f(z)}$ for all $\overbar F_i, \overbar F_j$ and all $z \in \overbar F_i\cap \overbar{F}_j$,   
		\item and if more than two codimension one faces meet at a point, then the angle between any pair of inner normal vectors of faces of $N$ is equal to the angle of the inner normal vectors of the corresponding faces of  $M$.
		\item If we furthermore assume that $M$ admits a top\footnote{We say a face $F^t$ of $M$ is a top if $F^t$ is contained in a leaf $X\times\{r_0\}$, and every point in $M$ connects to the top via a geodesic along $\partial_r$-direction.}, then the metric $\overbar g$ of $N$ is also hyperbolic, and the restriction of $\overbar g$ on the face of $N$ that maps to the top of $M$ is flat.
	\end{enumerate}  Moreover, the following hold. 
	\begin{enumerate}[label=$(\mathrm{\Roman*})$]
		\item For  $x$  in a codimension one face $\overbar F$ of $N$, if
		$$A>\langle\nu,\nabla r\rangle,$$
		at $x$, then $\partial f$ is an isometry at $x$.
		\item Assume that $M$ has a top and in addition $N$ is simply connected. In this case, $N$ embeds isometrically into the hyperbolic space $\mathbb H^n$. Then there exists an isometry $\vartheta$ of $\mathbb H^n$ such that,  for all $x$ in the interior of codimension one faces of $N$, $\vartheta$ maps the unit inner normal vector at $x\in\partial N$ to the unit inner normal vector of $f(x)$, provided that $f(x)$ is also in the interior of codimension one faces of $M$. 
	\end{enumerate}
\end{theorem} 

In particular, Theorem \ref{thm:hyperbolicComparison-intro} applies the following class of polytopes in hyperbolic spaces.  
%

Recall that  $M=P\times [0,1]$ is called a parabolic prism in the hyperbolic space $\mathbb H^n$ if $P$ is a convex polyhedron with Euclidean metric $h$, and $M$ is equipped with the hyperbolic metric $dr^2+e^{2r}h$, cf. Definition \ref{def:prism} below.  Note that, both the top face $P\times\{0\}$ and bottom face $P\times\{1\}$ are horospheres in this case. 
\begin{theorem}\label{thm:hyperbolicPrismComparisonIntro}
	Let $(M, g)$ be a parabolic prism  in the hyperbolic space $\mathbb H^n$ and $f\colon (N, \overbar g)\to (M, g)$ a spin polytope map between $n$-dimensional manifolds with polyhedral boundary. 
Assume that the degree of $f$ is non-zero. If
\begin{enumerate}[label=$(\arabic*)$]
	\item 
	$\Sc(\overbar{g})_x \geq \Sc(g)_{f(x)} = -n(n-1)$ for all $x\in N$, 
	\item $H_{\overbar{g}}(\overbar{F}_i)_y \geq  H_{g}(F_i)_{f(y)}$ for all $y$ in each codimension one face $\overbar{F}_i$ of $N$, 
	\item $\theta_{ij}(\overbar{g})_z\leq \theta_{ij}(g)_{f(z)}$ for all $\overbar F_i, \overbar F_j$ and all $z \in \overbar F_i\cap \overbar{F}_j$,   
\end{enumerate} 
then \begin{enumerate}[label=$(\roman*)$]	
	\item 
	$\Sc(\overbar{g})_x = \Sc(g)_{f(x)}$ for all $x\in N$, 
	\item $H_{\overbar{g}}(\overbar{F}_i)_y =  H_{g}(F_i)_{f(y)}$ for all $y$ in each codimension one face $\overbar{F}_i$ of $N$, 
	\item $\theta_{ij}(\overbar{g})_z =  \theta_{ij}(g)_{f(z)}$ for all $\overbar F_i, \overbar F_j$ and all $z \in \overbar F_i\cap \overbar{F}_j$, 
	\item furthermore, $N$ is also a parabolic  prism in $\mathbb H^n$, and  is locally isometric to  $M$.
\end{enumerate} 
\end{theorem}

In dimension $\leq 7$,  Li obtained a  version of the above theorem in some special cases via the minimal surface method \cite{MR4158684}. We emphasize that Theorem \ref{thm:hyperbolicPrismComparisonIntro} not only asserts that $N$ is a parabolic prism in $\mathbb H^n$, but also shows $N$ is locally isometric to $M$. This latter assertion does not seem to be accessible by the minimal surface method. 

Note that, if $N$ is locally isometric to $M$, then  all angles of $N$ (including angles between non-adjacent faces at a vertex)
coincide with the corresponding angles of $M$. In particular, as a consequence of Theorem \ref{thm:hyperbolicPrismComparisonIntro}, we have the following theorem, which is an analogue of Stoker's conjecture \cite{MR222765} for parabolic prisms in hyperbolic spaces. 

\begin{theorem}\label{thm:stokeranalogue-intro}
	Suppose $M_1$ and $M_2$ are two parabolic prisms in $\mathbb H^n$ of the same combinatorial type such that all corresponding dihedral angles are equal, then all corresponding angles of $M_1$ and $M_2$ are equal. 
\end{theorem}
 We emphasize that the above theorem asserts that all corresponding angles (including the face angles\footnote{Face angles of any given face are dihedral angles of that face, when the face is viewed as a polyhedron itself. }  and the angles between \emph{non-adjacent} codimension one faces) of $M_1$ and $M_2$ coincide. 

Now we consider the case where the warping function $\varphi$ satisfies $\varphi'=0$, which corresponds to  the case of direct product metrics on $X\times I$. Our third main theorem of the paper states  the dihedral rigidity  continues to hold in this case as well. More precisely, we have the following theorem. 

\begin{theorem}\label{thm:rigidityOfProduct-intro}
	Let $(X,g_X)$ be an $(n-k)$-dimensional Riemannian manifold with polyhedral boundary such that 
	\begin{enumerate}[label=$(\alph*)$]
		\item the curvature operator of $g_X$ is non-negative,
		\item each codimension one face $F^X_i$ of $X$ is convex, that is, the second fundamental form of  $F^X_i$  is non-negative, 
		\item all dihedral angles $\theta_{ij}(g_X)$ of $X$ are $<\pi$, 
		\item the Ricci curvature of $g_X$ is positive everywhere on $X$. 
	\end{enumerate} 
	Suppose that $Y$ is a $k$-dimensional convex domain of $\mathbb R^k$ equipped with the flat metric.  Let $(M, g)$ be $X\times Y$ equipped  the product metric 
	$$g=g_X+dy_1^2+\cdots+dy_k^2.$$
	Let $(N,\overbar g)$ be an $n$-dimensional manifold with polyhedral boundary and $f\colon N\to M$  a distance non-increasing spin polytope map. If
	\begin{enumerate}[label=$(\arabic*)$]
		\item 
		$\Sc(\overbar{g})_x \geq \Sc(g)_{f(x)}$ for all $x\in N$, 
		\item $H_{\overbar{g}}(\overbar{F}_i)_y \geq  H_g(F_i)_{f(y)}$ for all $y$ in each codimension one face $\overbar{F}_i$ of $N$, 
		\item $\theta_{ij}(\overbar{g})_z\leq \theta_{ij}(g)_{f(z)}$ for all $\overbar F_i, \overbar F_j$ and all $z \in \overbar F_i\cap \overbar{F}_j$, and
		\item $\chi(X)\ne 0$ and $\deg(f) \neq 0$, 
	\end{enumerate} 
	then we have 
	\begin{enumerate}[label=$(\roman*)$]	
		\item 
		$\Sc(\overbar{g})_x = \Sc(g)_{f(x)}$ for all $x\in N$, 
		\item $H_{\overbar{g}}(\overbar{F}_i)_y =  H_{g}(F_i)_{f(y)}$ for all $y$ in each codimension one face $\overbar{F}_i$ of $N$, 
		\item $\theta_{ij}(\overbar{g})_z =  \theta_{ij}(g)_{f(z)}$ for all $\overbar F_i, \overbar F_j$ and all $z \in \overbar F_i\cap \overbar{F}_j$, 
		\item furthermore,  $N$ is isometric to a direct product space $(\overbar X\times \overbar Y, g_{\overbar X}\times g_{\overbar Y})$ and the map $f$ respects the product structure such that
		\begin{enumerate}[label=$(\mathrm{\Roman*})$]
			\item  $f|_{\overbar X}\colon (\overbar  X, g_{\overbar X}) \to (X, g)$ is a local isometry, and
			\item $\overbar Y$ is also a convex subspace in $\R^k$ such that, after possibly one orthogonal transform of $\mathbb R^k$, the unit inner normal vectors of $y\in \partial \overbar Y$ and $f(y)\in \partial Y$ are equal for all $y\in \partial \overbar Y$.
		\end{enumerate} 
	\end{enumerate} 
\end{theorem}

The paper is organized as follows. In Section \ref{sec:radconvexrigidity}, we prove the first main theorem (Theorem \ref{thm:rigidityOfWarpedProductintro}) of the paper, a dihedral extremality and rigidity theorem for a large class of submanifolds with polyhedral boundary of warped product spaces. In Section \ref{sec:hyperbolic}, by combining the methods developed by the authors in  \cite{Wang2022:fl}, we extend the techniques in Section \ref{sec:radconvexrigidity} to prove a dihedral extremality and rigidity theorem for a large class of hyperbolic manifolds. As a special case, we obtain   a dihedral extremality and rigidity theorem for all parabolic prisms in hyperbolic spaces. In Section  \ref{sec:directprod}, we prove a dihedral rigidity theorem for  direct products of flat spaces and certain non-negatively curved spaces.

\section{Dihedral rigidity for radially convex domains of warped product manifolds}\label{sec:radconvexrigidity}

In this section, we prove Theorem \ref{thm:rigidityOfWarpedProductintro}: a dihedral extremality and rigidity theorem for  radially convex domains (cf. Definition \ref{def:radiallyConvex}) in warped product manifolds. 

\subsection{Manifolds with polyhedral boundary}

In this subsection, we briefly review the definition of \emph{manifolds with polyhedral  boundary}, which are locally modeled on $n$-dimensional polyhedra in $\mathbb R^n$. For a given Hausdorff space $X$, a polytope chart $(U, \varphi)$  for $X$ is a homeomorphism $\varphi$ from an open subset $U$ of $M$ to an open subset of an $n$-dimensional polyhedron in $\mathbb R^n$.  Two polytope charts $(U_1, \varphi_1)$ and $(U_2, \varphi_2)$ are $C^\infty$-related if either $U_1\cap U_2$ is empty or the map 
\[ 	\varphi_2\circ \varphi_1^{-1}\colon \varphi_1(U_1\cap U_2) \to \varphi_2(U_1\cap U_2) \]
is a diffeomorphism (of open subsets of $n$-dimensional polyhedra). Again, a system of pairwise $C^\infty$-related charts of $X$ that covers $X$ is called an atlas of $X$. 
\begin{definition}\label{def:polytopeboundary}
	A smooth manifold with polyhedral  boundary is a Hausdorff space equipped with a maximal atlas of polytope charts. 
\end{definition}

A Riemannian manifold with polyhedral boundary is a smooth manifold with polyhedral boundary equipped with a smooth Riemannian metric.  

\begin{definition}
	Let $N$ be an $n$-dimensional manifold with polyhedral boundary. We define the codimension $k$ stratum of $N$ to be the set of interior points of all codimension $k$ faces of $N$. 
\end{definition}

For each point $x$ in the codimension $k$ stratum of $N$, it admits a small neighborhood $U$ of the form: 
\[ \R^{n-k} \times P\]
such that $P$ is a polyhedral corner in $\R^k$ enclosed by hyperplanes passing through the origin of $\R^k$ and $x$ is the origin of $\R^n$.  In this case, we call the partial derivatives along $\mathbb R^{n-k}$ the base directions of the neighborhood $U$ of $x$.

\begin{definition}\label{def:polytopeMap} A  map $f\colon (N,\overbar g)\to (M,g)$ between Riemannian manifolds with polyhedral boundary is called a \emph{polytope map} if 
	\begin{enumerate}
		\item 	$f$ is Lipschitz,\footnote{Here $f\colon (N,\overbar g)\to (M,g)$ is said to be Lipschitz if there exists $C>0$ such that $d_M(f(x), f(y))\leq C\cdot d_N(x, y)$  for all points in $x, y\in N$.}
		\item $f$ is smooth away from the codimension three faces of $N$,
		\item $f$ maps the codimension $k$ stratum of $N$ to the  codimension $k$ stratum of $M$, and 
		\item every point $x$ in $N$ has a small open neighborhood $U$ such that $f$ is smooth with respect to the base directions on $U$.  
	\end{enumerate}
\end{definition}

\subsection{Comparisons on scalar curvature and mean curvature} \label{subsec:2.1}
In this subsection, we prove some key estimates for manifolds with polyhedral boundary that will be used in the proofs of our main theorems.

\begin{setup}\label{setup:warpedProduct}
	Let $(X, h)$ be an $(n-1)$-dimensional Riemannian manifold and $X\times I$  equipped with a warped product metric $g$ of the following form: 
	\begin{equation}\label{eq:warp}
	g = dr^2 + \varphi^2 h
	\end{equation}
	where $r$ is the parameter for the interval $I$ and $\varphi$ is a positive smooth function on $I$. Furthermore, $\varphi'\geq 0$ and $\varphi$ is log-concave.
	Let $M$ be a radially convex domain of $(X\times I, g)$ in the sense of Definition \ref{def:radiallyConvex}. 
	Suppose $(N, \overbar g)$ is a Riemannian manifold with polyhedral boundary and $f\colon (N, \overbar g) \to (M, g)$ is a spin polytope map.
\end{setup}

 Let $(M,g)$ be a region in $(X\times I,dr^2+\varphi^2 h)$. Note that the warped product metric is conformally a product metric after a reparametrization. More precisely, if we consider a smooth function $\rho$ such that
$$\rho'=\varphi(\rho),$$
then we have
$$dr^2+\varphi^2 h=\varphi(\rho(s))^2\big(ds^2+h\big),$$
where $s=\rho^{-1}(r)$. Let $A$ and $\widetilde A$ be the second fundamental forms of $\partial M$ with respect to the warped product metric $g=dr^2+\varphi^2 h$, and the conformal metric $\varphi^{-2}g=ds^2+h$, respectively. Direct computation shows that
$$\widetilde A(X,Y)=A(X,Y)-\frac{\varphi'}{\varphi}\langle\nu,\partial_r\rangle\langle X,Y\rangle_g$$
for any boundary tangent vectors $X,Y$. Therefore, we arrive at an equivalent characterization of the condition in Definition \ref{def:radiallyConvex}.
\begin{lemma}
	Let $(M,g)$ be a region in $(X\times I,dr^2+\varphi^2 h)$. The following are equivalent.
	\begin{enumerate}
		\item $M$ is radially convex.
		\item $(M,\varphi^{-2}g)$ has non-negative second fundamental form.
	\end{enumerate}
\end{lemma}

 Let $S_N$ (resp. $S_M$) be the local spinor bundle over $N$ (resp. $M$). In general, the bundles $S_N$ and $S_M$ only exist locally. However, since $f\colon N \to M$ is a spin  map, the bundle $TN\oplus f^\ast TM$ admits a spin structure. Let us denote by $S_N\otimes f^\ast S_M \coloneqq S_{TN\oplus f^\ast TM}$ the associated spinor bundle over $N$. We denote by $\ncon$ (resp. $\mcon$)  the Levi--Civita connection on $TN$ (resp. $TM$), and by $\nabla^{S_N}$ (resp. $\nabla^{S_M}$)  the associated connection on $S_N$ (resp. $S_M$). 
 We denote by $D$ the Dirac operator on $S_N\otimes f^*S_M$ over $N$ associated to the connection 
 $$\nabla\coloneqq \nabla^{S_N}\otimes 1+1\otimes f^*(\nabla^{S_M}).$$ Throughout the paper,  we will denote the Clifford multiplication of a vector $\overbar v\in TN$ by $\overbar c(\overbar v)$ and the Clifford multiplication of a vector $v\in  f^\ast TM$ by $c(v)$. 
 
Now let us introduce the following new connection on $S_M$:
\begin{equation}\label{eq:nablaHat}
\tcon^{S_M}_{e_i} =  \nabla_{e_i}^{S_M} + \frac{1}{2} c(\prescript{M}{}\nabla_{e_i} \partial_r) c(\partial_r) 
\end{equation}
 It is easy to verify that  
\[ \tcon_{e_i}^{S_M} c(\partial_r) = c(\partial_r) \tcon_{e_i}^{S_M}  \]
for all $e_i \in TM$. 

We have the following new connection 
\begin{equation}\label{eq:hatcon}
\tcon\coloneqq \nabla^{S_N}\otimes 1+1\otimes f^*(\tcon^{S_M})
\end{equation}
on the bundle $S_N\otimes f^*S_M$ over $N$. Let $\tdirac$ be the Dirac operator on $S_N\otimes f^*S_M$ associated to the new connection $\tcon$, that is,
\[ \tdirac\coloneqq \sum_{i} \overbar c(\overbar e_i) \tcon_{\overbar e_i}.\]

Let $\mathcal P\colon C^\infty(N,S_N\otimes f^*S_M)\to C^\infty(N,T^*N\otimes S_N\otimes f^*S_M)$ be the Penrose operator defined by
\begin{equation}\label{eq:penrosedef}
\mathcal P_{\overbar e_i}\sigma\coloneqq \tcon_{\overbar e_i}\sigma+\frac{1}{n}\overbar c(\overbar e_i)\tdirac\sigma
\end{equation}
for all $\overbar e_i\in TN$ and all $\sigma\in C^\infty(N,S_N\otimes f^*S_M)$.
We have the following useful identity (cf. \cite[Section 5.2]{spinorialapproach}):
\begin{equation}\label{eq:penrose}
|\tcon\sigma|^2=|\mathcal P\sigma|^2+\frac 1 n|\tdirac \sigma|
\end{equation}
all $\sigma\in C^\infty(N,S_N\otimes f^*S_M)$.

Let $\overbar\epsilon$ and $\epsilon$ be the grading operator on $S_N$ and $S_M$ respectively. Then the operator $\overbar\epsilon\otimes\epsilon$ defines a $\Z_2$-grading on $S_N\otimes f^*S_M$. The operator that is most relevant to our current geometric setup is the following Dirac operator (with a potential) acting on  $S_N\otimes f^*S_M$ over $N$: 
\begin{equation}\label{eq:Diracpotential}
\tdirac_\Psi\coloneqq \tdirac + \Psi.
\end{equation}
Here the potential $\Psi$ is given by
\begin{eqnarray}\label{eq:potential}
\Psi\coloneqq \frac{n}{2}f^*(\frac{\varphi'}{\varphi})\cdot (\overbar\epsilon\otimes\epsilon)\cdot (1
\otimes  c(\partial_r)),
\end{eqnarray}
where  $\varphi$ is the warping function appearing in the warped product metric from line \eqref{eq:warp} and $n = \dim M$.

We further fix some notations on the boundaries. For each codimension one face $\overbar F$ (resp. $F$) of $N$ (resp. $M$), let $\overbar e_n$ (resp. $e_n$) be its unit inner normal vector field. Let us define the boundary Clifford actions by setting 
\begin{equation*}
\overbar c_\partial(\overbar e_\lambda)\coloneqq\overbar c(\overbar e_n)\overbar c(\overbar e_\lambda) \textup{ and } c_\partial( e_\lambda)\coloneqq c( e_n) c( e_\lambda)
\end{equation*}
for $\overbar e_\lambda \in T\overbar F$ and $e_\lambda \in TF$. Similarly to line \eqref{eq:nablaHat}, we define the following boundary connections 
\begin{equation}\label{eq:nablaSNBoundary}
\nabla^{S_N,\partial}_{\overbar e_j}=\nabla^{S_N}_{\overbar e_j}+\frac 1 2 \overbar c(\prescript{N}{}\nabla_{\overbar e_j} \overbar e_n)\overbar c(\overbar e_n),
\end{equation}
\begin{equation}\label{eq:nablaSMBoundary}
\nabla^{S_M,\partial}_{e_j}=\nabla^{S_M}_{e_j}+\frac 1 2 c(\prescript{M}{}\nabla_{e_j}  e_n)c( e_n),
\end{equation}
and
\begin{equation}\label{eq:boundaryconnection}
\nabla^\partial=\nabla^{S_N,\partial}\otimes 1+1\otimes f^*(\nabla^{S_M,\partial}).
\end{equation}
In particular, a direct computation shows that the Clifford multiplication $\overbar c(\overbar e_n)$ (resp. $c(e_n)$ ) is parallel with respect to $\nabla^{S_M,\partial}$ (resp. $\nabla^{S_N,\partial}$). We define  the boundary Dirac operator $D^\partial$  to be 
\begin{equation}\label{eq:boundarydirac}
D^{\partial}\coloneqq \sum_{j=1}^{n-1}\overbar c_\partial(\overbar e_j)\nabla^{\partial}_{\overbar e_j}.
\end{equation}


\begin{definition}\label{def:boundarycondition}
	A section $\sigma$ of $S_N\otimes f^* S_M$ over $N$ is said to satisfy the local boundary condition $B$  if $\sigma|_{\overbar F_i}$ lies in $\ker(1+(\overbar \epsilon\otimes\epsilon)(\overbar c(\overbar e_n)\otimes c(e_n)))$ on every codimension one face $\overbar F_i$ of $N$, where $\overbar \epsilon$ and $\epsilon$ are the grading operators on $S_N$ and $f^\ast S_M$.
\end{definition}

Now let us prove the following key estimate on comparisons of scalar curvature and  mean curvature. 
\begin{proposition}\label{prop:comparison}
	Assume  Geometric setup \ref{setup:warpedProduct} and in particular $M$ is a radially convex domain of $(X\times I, g)$. If the curvature operator of $(X, h)$ is non-negative and $f\colon (N, \overbar g) \to (M, g)$ is distance non-increasing, then  we have 
	\begin{equation}\label{eq:comparison}
	\begin{aligned}
	\int_N|\tdirac_\Psi\sigma|^2\geq& \int_N |\mathcal P\sigma|^2 + \frac{n}{4(n-1)}\int_N(\Sc(\overbar g)-f^*\Sc(g))|\sigma|^2\\
	&+\frac{n}{2(n-1)}\sum_{i}\int_{\partial \overbar F_i}(H_{\overbar g}(\overbar F_i)-f^*H_g(F_i))|\sigma|^2
	\end{aligned}\end{equation}
	for every smooth section  $\sigma$ of $S_N\otimes f^* S_M$ that satisfies the boundary condition $B$ in Definition \ref{def:boundarycondition}, and vanishes near all codimension $2$ singularities.
\end{proposition}

\begin{proof}
	Without loss of generality, we assume that both $N$ and $M$ are even-dimensional. The computation for the odd-dimensional case\footnote{Alternatively, for applications in later sections, we can reduce the odd dimensional case to the even dimensional case by taking  direct product with the unit interval.} is completely similar.

	Note that $\overbar\epsilon\otimes\epsilon c(\partial_r)$ is self-adjoint and anti-commutes with $\tdirac$, and its square is the identity operator. Let us write for short
	\begin{equation*}
	\Psi=f^\ast \psi\cdot (\overbar\epsilon\otimes\epsilon)\cdot (1
	\otimes  c(\partial_r)), \textup{ where } \psi= \frac{n\varphi'}{2\varphi}.
	\end{equation*}
	By the definition of $\tdirac_\Psi$ from line \eqref{eq:Diracpotential}, we have the following pointwise equality
	\begin{equation}\label{eq:Bsquare}
	\langle \tdirac_\Psi\sigma, \tdirac_\Psi \sigma\rangle
	=|\tdirac\sigma|^2+\langle \Psi\sigma,\tdirac\sigma\rangle+\langle\tdirac\sigma, \Psi\sigma\rangle+(f^\ast\psi)^2 |\sigma|^2.
	\end{equation}
	By the Stokes theorem, we have
	\begin{equation}\label{eq:StokesDsquare}
	\int_{N} |\tdirac\sigma|^2=\int_{N} \langle\tdirac^2\sigma,\sigma\rangle+\int_{\partial N} \langle\overbar c(\overbar e_n) \tdirac\sigma,\sigma\rangle.
	\end{equation}
Note that 
	\begin{equation}\label{eq:Lichne}
	\tdirac^2=\tcon^*\tcon+\mathcal R, 
	\end{equation}
where $\mathcal R$ is the corresponding  curvature endomorphism of $S_N\otimes f^*S_M$. 
	Again by the Stokes theorem, we have
	\begin{equation}\label{eq:Stokesnabla}
	\int_{N}\langle \tcon^*\tcon\sigma,\sigma\rangle=\int_{N}| \tcon \sigma|^2 +\int_{\partial N}\langle\tcon_{\overbar e_n}\sigma,\sigma\rangle,
	\end{equation} 
	where $\overbar e_n$ is the inner unit normal vector of $\partial N$.
	
	By combining line \eqref{eq:StokesDsquare}, \eqref{eq:Lichne}, \eqref{eq:Stokesnabla} and \eqref{eq:penrose}, we obtain that
	\begin{equation}\label{eq:D^2}
	\begin{split}
	\int_{N} |\tdirac\sigma|^2=&\frac{n}{n-1}\int_{N} |\mathcal P\sigma|^2+\frac{n}{n-1}\int_{N}\langle\mathcal R\sigma,\sigma\rangle\\
	&+\frac{n}{n-1}\int_{\partial N}\langle (\overbar c(\overbar e_n)\tdirac+\tcon_{\overbar e_n})\sigma,\sigma\rangle
	\end{split}
	\end{equation}
	By the Stokes theorem, the second and third terms on the right hand side of the equation from  line \eqref{eq:Bsquare} give 
	\begin{equation}\label{eq:2and3}
	\begin{split}
	&\int_{N} \langle \Psi \sigma,\tdirac\sigma\rangle+\langle\tdirac\sigma,\Psi\sigma\rangle\\
	=&\int_{N} \langle\tdirac \Psi \sigma,\sigma\rangle+
	\langle\Psi \tdirac\sigma,\sigma\rangle+
	\int_{\partial N} \langle\overbar c(\overbar e_n) \Psi \sigma,\sigma\rangle\\
	=&\int_{N} \langle [\tdirac,\Psi] \sigma,\sigma\rangle
	+\int_{\partial N} \langle\overbar c(\overbar e_n) \Psi\sigma,\sigma\rangle.
	\end{split}
	\end{equation}
	Note that $[\tdirac,\Psi]=\overbar c(\grad \psi) \cdot (\overbar\epsilon\otimes\epsilon)\cdot (1
	\otimes  c(\partial_r)) $, where $\grad \psi$ is the gradient of $\psi$. 
	
	To summarize, we have
	\begin{equation}\label{eq:DPsiSquare}
	\begin{split}
	\int_{N} |\widehat D_\Psi\sigma|^2=&\frac{n}{n-1}\int_{N} |\mathcal P\sigma|^2\\
	+&\int_{N}\frac{n}{n-1}\langle\mathcal R\sigma,\sigma\rangle
	+ \langle \overbar c(\grad(f^\ast\psi)) \cdot (\overbar\epsilon\otimes\epsilon c(\partial_r)) \sigma,\sigma\rangle+(f^\ast\psi)^2|\sigma|^2\\
	&+\int_{\partial N}\frac{n}{n-1}\langle (\overbar c(\overbar e_n)\tdirac+\tcon_{\overbar e_n})\sigma,\sigma\rangle+\langle\overbar c(\overbar e_n) \Psi\sigma,\sigma\rangle
	\end{split}
	\end{equation}
	where $\grad(f^\ast \psi)$ is the gradient of $f^\ast \psi$. 
	We first evaluate the curvature term in the second line of Equation \eqref{eq:DPsiSquare}. Write $\overbar r=f^* r$, which is a $1$-Lipschitz function over $N$. Hence we have
	\begin{align*}
	[\tdirac,\Psi]=&\overbar c(\grad(f^\ast \psi)) \cdot (\overbar\epsilon\otimes\epsilon)\cdot (1
	\otimes  c(\partial_r))\\
	=&  -f^\ast\Big(\frac{d\psi}{dr}\Big)\cdot 
	(\overbar\epsilon\otimes\epsilon) (\overbar c(\grad\overbar r))\otimes c(\partial_r)).
	\end{align*}
	Since $\psi'\leq 0$, we obtain that
	\begin{equation}\label{eq:gradpsi}
	\langle[\tdirac,\Psi]\sigma,\sigma\rangle\geq f^\ast\Big(\frac{d\psi}{dr}\Big) |\sigma|^2
	=f^\ast\Big(\big(n\varphi'/2\varphi\big)'\Big)\cdot |\sigma|^2
	\end{equation}
For the warped product metric $g = dr^2 + \varphi^2 h$ on $X\times I$, its scalar curvature is given in terms of the scalar curvature of $h$ as follows:  
	\begin{equation}\label{eq:warpscalar}
	\Sc(g)=\frac{\Sc(h)}{\varphi^2}-\frac{(n-1)(n-2)(\varphi')^2}{\varphi^2}-\frac{2(n-1)\varphi''}{\varphi},
	\end{equation}
	which is equivalent to the following
	\begin{equation}\label{eq:scalarcurvatureofwarped}
	\frac{n}{n-1}\cdot \frac{\Sc(g)}{4}=\frac{n}{n-1}\cdot \frac{\Sc(h)}{4\varphi^2}-\Big(\frac{n\varphi'}{2\varphi}\Big)^2
	-\frac{d}{dr}\Big(\frac{n\varphi'}{2\varphi}\Big).
	\end{equation}
	The above discussion together with Lemma \ref{lemma:curvature>=} below implies that
	\begin{equation}\label{eq:Sc-Sc}
	\begin{split}
	&\frac{n}{n-1}\langle\mathcal R\sigma,\sigma\rangle
	+(f^\ast \psi)^2|\sigma|^2+ \langle [\tdirac,\Psi] \sigma,\sigma\rangle\\
	\geq& \frac{n}{4(n-1)}(\Sc(\overbar g)-f^*\Sc(g))|\sigma|^2.
	\end{split}
	\end{equation}
	
	Now we consider the boundary term in the third line of Equation \eqref{eq:DPsiSquare} on each face, that is, 
	\begin{equation}\label{eq:boundary}
		\frac{n}{n-1}\langle (\overbar c(\overbar e_n)\tdirac+\tcon_{\overbar e_n})\sigma,\sigma\rangle+\langle\overbar c(\overbar e_n) \Psi\sigma,\sigma\rangle.
	\end{equation}
	
	Let $\partial_r$ be the gradient of the function $r$. Denote  $\alpha=\langle e_n,\partial_r\rangle$. We then have an orthogonal decomposition
	$$\partial_r=\alpha e_n+\beta,$$
	where $\beta$ lies in the tangent space of the face. For brevity, let us write
	\begin{equation*}
		\gamma\coloneqq (\overbar\epsilon\otimes\epsilon)(\overbar c(\overbar e_n)\otimes c(e_n)).
	\end{equation*}
	Recall that $\sigma$ satisfies the boundary condition $B$, that is,  $\sigma=-\gamma\sigma$.
	By definition, the second term of line \eqref{eq:boundary} is given by
	\begin{align*}
		\langle\overbar c(\overbar e_n) \Psi\sigma,\sigma\rangle=&\alpha f^*\psi\langle\overbar c(\overbar e_n) (\overbar\epsilon\otimes\epsilon)(1\otimes c(e_n))\sigma,\sigma\rangle+f^*\psi\langle\overbar c(\overbar e_n) (\overbar\epsilon\otimes\epsilon)(1\otimes \beta)\sigma,\sigma\rangle\\
		=&\alpha f^*\psi|\sigma|^2,
	\end{align*}
	where the last equality follows from the boundary condition, and the fact that $\gamma$ anti-commutes with $c(\overbar e_n) (\overbar\epsilon\otimes\epsilon)(1\otimes \beta)$ as $\beta$ is orthogonal to $e_n$.
	
	Now we consider the first term in line \eqref{eq:boundary}.	
	Let $\mathcal A$ be the boundary curvature endomorphism, that is,
	\begin{equation}\label{eq:mathcalA}
	\mathcal A\coloneqq D-D^\partial
	\end{equation}
	where $D^\partial$ is given as in line \eqref{eq:boundarydirac}. The endomorphism $\mathcal A$ is given by
		\begin{equation}\label{eq:boundaryA}
		\mathcal A=\frac{H(\overbar g)}{2}-\frac{1}{2}\sum_{\lambda,\mu}A(f_*\overbar e_\lambda, e_\mu)\overbar c_\partial(\overbar e_\lambda)\otimes c_\partial(e_\mu).
	\end{equation}	
	By definition of $\tcon$ from line \eqref{eq:hatcon}, we have
	\begin{equation}\label{eq:cenD+nablaen}
	\begin{split}
	&\overbar c(\overbar e_n)\tdirac+\tcon_{\overbar e_n}=
	\overbar c(\overbar e_n)\sum_\lambda \overbar c(\overbar e_\lambda)\tcon_{\overbar e_\lambda}\\
	=&D^\partial+\mathcal A+\frac 1 2\sum_\lambda \overbar c(\overbar e_n)\overbar c(\overbar e_\lambda)\otimes c(\mcon_{f_*\overbar e_\lambda} \partial_r)c(\partial_r)
	\end{split}
	\end{equation}
A direct computation shows that  $\langle D^\partial\sigma,\sigma\rangle=0$ if $\sigma$ satisfies the boundary condition $B$ (cf. \cite[Lemma 3.3]{Wang:2021tq}).

We denote
$$c(f_*\overbar e_\lambda\wedge\partial_r)\coloneqq\frac 1 2(f_*\overbar c(e_\lambda)c(\partial_r)-c(\partial_r)c(f_*\overbar e_\lambda).$$
A direct computation shows that the last term in line \eqref{eq:cenD+nablaen} applied on section $\sigma$ is given by
 \begin{align*}
 	&\frac 1 2\langle\sum_\lambda \overbar c(\overbar e_n)\overbar c(\overbar e_\lambda)\otimes c(\mcon_{f_*\overbar e_\lambda} \partial_r)c(\partial_r)\rangle\\
= &\frac 1 2f^*(\frac{\varphi'}{\varphi})\langle\sum_\lambda \overbar c(\overbar e_n)\overbar c(\overbar e_\lambda)\otimes c(f_*\overbar e_\lambda\wedge\partial_r)\sigma,\sigma\rangle\\
=&\frac 1 2f^*(\frac{\varphi'}{\varphi})\cdot\alpha\langle\sum_\lambda \overbar c(\overbar e_n)\overbar c(\overbar e_\lambda)\otimes c(f_*\overbar e_\lambda\wedge e_n)\sigma,\sigma\rangle+\frac 1 2f^*(\frac{\varphi'}{\varphi})\langle\sum_\lambda \overbar c(\overbar e_n)\overbar c(\overbar e_\lambda)\otimes c(f_*\overbar e_\lambda\wedge\beta)\sigma,\sigma\rangle\\
=&-\frac 1 2f^*(\frac{\varphi'}{\varphi})\cdot\alpha\langle\sum_\lambda \overbar c_\partial(\overbar e_\lambda)\otimes c_\partial(f_*\overbar e_\lambda)\sigma,\sigma\rangle
 \end{align*}
where the last equality follows from the boundary condition $\gamma\sigma=-\sigma$, and that $f$ sends $\partial N$ to $\partial M$.

To summarize, we have shown that line \eqref{eq:boundary} is now given by
	\begin{equation}\label{eq:boundary2}
		\frac{n}{n-1}\langle (\overbar c(\overbar e_n)\tdirac+\tcon_{\overbar e_n})\sigma,\sigma\rangle+\langle\overbar c(\overbar e_n) \Psi\sigma,\sigma\rangle=\frac{n}{n-1}\langle\widehat{\mathcal A}\sigma,\sigma\rangle,
\end{equation}
where
\begin{equation}\label{eq:hatA}
	\begin{split}
		\widehat{\mathcal A}=\frac{H(\overbar g)}{2}-\frac{1}{2}\sum_{\lambda,\mu}\Big(A(f_*\overbar e_\lambda, e_\mu)+\alpha f^*\psi\Big)\overbar c_\partial(\overbar e_\lambda)\otimes c_\partial(e_\mu)+\frac{n-1}{2}\alpha f^*(\frac{\varphi'}{\varphi}).
	\end{split}
\end{equation}
By Lemma \ref{lemma:secondff>=}, we obtain from line \eqref{eq:DPsiSquare} the following inequality 
\begin{equation*}
	\begin{aligned}
		\int_{N} |\widehat D_\Psi\sigma|^2 \geq &\frac{n}{n-1}\int_{N} |\mathcal P\sigma|^2  + \frac{n}{4(n-1)}\int_{N} (\Sc(\overbar g)-f^*\Sc(g))|\sigma|^2
		\\
		&+\frac{n}{2(n-1)}\int_{\partial N}(H(\overbar g) -f^*H(g))|\sigma|^2 
\end{aligned}
\end{equation*}  
This finishes the proof.
\end{proof}

Now let us prove the estimates on scalar curvature and mean curvature that were used in the proof of Proposition \ref{prop:comparison} above. 
\begin{lemma}\label{lemma:curvature>=}
	If the curvature operator on $(X, h)$ is non-negative, then
	\begin{equation}\label{eq:curvature>=}
	\mathcal R\geq \frac{\Sc(\overbar g)}{4}-f^*\Big(\frac{\Sc(h)}{4\varphi^2}\Big),
	\end{equation}
where $\mathcal R$ is the curvature term from line \eqref{eq:Lichne}.
\end{lemma}
\begin{proof}
	For $2$-forms of $N$, we define the Clifford multiplication by
	\begin{equation}\label{eq:c(2-form)}
	\overbar c(\overbar e_i\wedge \overbar e_j)=\overbar c(\overbar e_i)\overbar c(\overbar e_j),
	\end{equation}
	where $\overbar e_i, \overbar e_j \in TN$ are orthogonal. The Clifford  multiplication $c(w)$ for a $2$-form $w$ over $M$ is defined similarly. 
	
	 By the Bochner--Lichnerowicz--Weitzenbock formula, we have
	\begin{equation}\label{eq:BLW}
	\mathcal R =\frac{\Sc(\overbar g)}{4}-\frac{1}{2}\sum_{i,j}\langle \widehat R f_*\overbar w_j,w_i\rangle_M \, \overbar c(\overbar w_j)\otimes c(w_i),
	\end{equation}
	where $\{\overbar w_j\}$ is a local orthonormal basis of $2$-forms on $N$, and $\{w_i\}$ is a local orthonormal basis of leaf-wise $2$-forms on $M$, and $\widehat R$ is the leaf-wise curvature form along each leaf $M\cap (X\times \{r\})$ of $M$. The reason that $\widehat R$ (instead of the usual curvature form of $M$) appears in the above formula is that $\partial_r$ is parallel with respect to the connection $\tcon$ in line \eqref{eq:hatcon}. 
	
	As the curvature operator $\widehat R$ is non-negative along each leaf, there exists a self-adjoint $L\in \Endo(\Bigwedge^2 TX)$  such that $\widehat R=L^2$, that is,
	$\langle \widehat R  w_j,w_i\rangle_{M}=\langle L  w_j,L w_i\rangle_{M}.$
	
	Set $$\overbar L w_k\coloneqq \sum_i\langle L w_k,f_*\overbar w_i\rangle_M\overbar w_i\in\Bigwedge^2 TN.$$
	The second term on the right hand side  of \eqref{eq:BLW} can be written as
	\begin{align*}
	&-\frac{1}{2}\sum_{i,j}\langle \widehat R f_*\overbar w_j,w_i\rangle_{M} \overbar c(\overbar w_j)\otimes c(w_i)\\
	=&-\frac{1}{2}\sum_{i,j,k}\langle L(f_*\overbar w_j),w_k\rangle_M \cdot 
	\langle Lw_i,w_k\rangle_M \cdot  \overbar c(\overbar w_j)\otimes c(w_i)\\
	=&-\frac{1}{2}\sum_{k}\overbar c(\overbar Lw_k)\otimes c(L w_k)\\
	=&\frac{1}{4}\sum_k\Big(\overbar c(\overbar Lw_k)^2\otimes 1+1\otimes c(L w_k)^2-\big(\overbar c(\overbar Lw_k)\otimes 1+1\otimes  c(L w_k)\big)^2\Big)\\
	\geq&\frac 1 4 \sum_k\overbar c(\overbar Lw_k)^2\otimes 1+\frac{1}{4}\sum_k 1\otimes c(L w_k)^2, 
	\end{align*}
	where the last inequality follows from the fact  that the element
	\[  \overbar c(\overbar Lw_k)\otimes 1+1\otimes  c(L w_k)\] is skew-symmetric, hence its square is non-positive.  
	
	
	The same proof for the Lichnerowicz formula (cf. \cite[Theorem II.8.8]{spingeometry}) shows that
	$$\sum_k  c(L w_k)^2=-  \frac{\Sc(\varphi^2h)}{2}=- \frac{\Sc(h)}{2\varphi^2}.$$
	Similarly, by the definition of $\overbar L$, we have
	\begin{align*}
	\sum_k\overbar c(\overbar Lw_k)^2 & =  \sum_{i,j,k}\langle \overbar Lw_k, f_*\overbar w_i\rangle_M \cdot 
	\langle \overbar Lw_k, f_*\overbar w_j\rangle_M \cdot \overbar c(\overbar w_i)\otimes c(\overbar w_j) \\
	& =\sum_{i,j}\langle \widehat R(f_*\overbar w_i),f_*(\overbar w_j)\rangle_M \cdot \overbar c(\overbar w_i)\overbar c(\overbar w_j).
	\end{align*}
	We choose a local $\overbar g$-orthonormal frame $\overbar e_1,\ldots,\overbar e_n$ of $TN$ and a local $g$-orthonormal frame $e_1,\ldots,e_n$ of $TM$ such that $f_*\overbar e_i=\mu_i e_i$ with $\mu_i\geq 0$. This can be done by diagonalizing $f^*g$ with respect to the metric $\overbar g$. Then we have 
	$f_*(\overbar e_i\wedge\overbar e_j)=\mu_i\mu_j e_i\wedge e_j$.
	As $f$ is distance-non-increasing, we have
	\begin{align*}
	\sum_k \overbar c(\overbar Lw_k)^2=- \sum_{i,j}\mu_i^2\mu_j^2(f^\ast \widehat R_{ijji})\geq -\frac{\Sc(\varphi^2h)}{2}=-\frac{\Sc(h)}{2\varphi^2}.
	\end{align*}
	This finishes the proof.
\end{proof}
\nocite{Lottboundary,Wang:2021tq}
\begin{lemma}\label{lemma:secondff>=}
	If the second fundamental form $A$ of $\partial M$ satisfies the inequality in Definition \ref{def:radiallyConvex}, then the operator $\widehat{\mathcal A}$ in line \eqref{eq:hatA} satisfies
	\begin{equation}\label{eq:H-H}
	\mathcal A\geq \frac{H(\overbar g)}{2}-\frac{f^*H(g)}{2}.
	\end{equation}
\end{lemma}
\begin{proof}
	The strategy is similar to that of Lemma \ref{lemma:curvature>=}. Denote
	$$\widehat A=A-\frac{\varphi'}{\varphi}\langle \nu,\partial_r\rangle.$$
	By assumption, $\widehat A$ is a non-negative symmetric form on each face of $\partial M$. Since we have assumed that $\nu=-e_n$, which is the unit outer normal vector of the face, we then have
	\begin{equation}\label{eq:hatA2}
	\widehat{\mathcal A}=\frac{H(\overbar g)}{2}-\frac{1}{2}\sum_{\lambda,\mu}\widehat A(f_*\overbar e_\lambda, e_\mu)\overbar c_\partial(\overbar e_\lambda)\otimes c_\partial(e_\mu)-\frac{n-1}{2}f^*\big(\frac{\varphi'}{\varphi}\langle \nu,\partial_r\rangle\big).
	\end{equation}
	
	Since $\widehat A\geq 0$, there exists a self-adjoint operator $L \in \Endo(TM)$ such that $\widehat A=L^2$, that is,
	$$ \widehat A(e_\lambda,e_\mu)=\langle L  e_\lambda,L e_\mu\rangle_{M}.$$
	Let us define $$\overbar L e_\nu\coloneqq \sum_\lambda\langle L e_\nu,f_*\overbar e_\lambda\rangle_M\cdot \overbar e_\lambda.$$
	
	We rewrite the second term in the right hand side of line \eqref{eq:hatA2} as follows.
	\begin{align*}
	&-\frac{1}{2}\sum_{\lambda,\mu}\widehat A(f_*\overbar e_\lambda, e_\mu)\overbar c_\partial(\overbar e_\lambda)\otimes c_\partial(e_\mu)\\
	=&-\frac{1}{2}\sum_{\lambda,\mu,\nu}\langle L(f_*\overbar e_\lambda),e_\nu\rangle_M 
	\langle L(e_\mu),e_\nu\rangle_M  \overbar c_\partial(\overbar e_\lambda)\otimes c_\partial(e_\mu)\\
	=&-\frac{1}{2}\sum_{\nu}\overbar c_\partial(\overbar Le_\nu)\otimes c_\partial(L e_\nu)\\
	=&\frac{1}{4}\sum_\nu\Big(\overbar c_\partial(\overbar Le_\nu)^2\otimes 1+1\otimes c_\partial(L e_\nu)^2-\big(\overbar c_\partial(\overbar Le_\nu)\otimes 1+1\otimes c_\partial(L e_\nu)\big)^2\Big)\\
	\geq&\frac 1 4 \sum_\nu \overbar c_\partial(\overbar Le_\nu)^2\otimes 1+\frac{1}{4}\sum_\nu 1\otimes  c_\partial(Le_\nu)^2,
	\end{align*}
	where the last inequality  follows from the fact  that the element
	\[  \big(\overbar c_\partial(\overbar Le_\nu)\otimes 1+1\otimes c_\partial(L e_\nu)\big) \] is skew-symmetric, hence its square is non-positive.  
	
	If we write $L e_\nu=\sum_\lambda L_{\nu\lambda} \cdot e_\lambda$, then we have
	\begin{align*}
	\sum_\nu c_\partial(Le_\nu)^2=&\sum_{\nu,\lambda}L_{\nu\lambda}L_{\nu\lambda}\cdot c_\partial(e_\lambda)^2+\sum_{\nu}\sum_{\lambda\ne\mu}L_{\nu\lambda}L_{\nu\mu}\cdot c_\partial(e_\lambda)c_\partial(e_\mu)\\
	=&-\sum_\lambda \widehat A_{\lambda\lambda}+\sum_{\lambda\ne\mu}\widehat A_{\lambda\mu} \cdot c_\partial(e_\lambda)c_\partial(e_\mu)\\
	=&-\tr(\widehat A)=- H(g)+(n-1)f^*\big(\frac{\varphi'}{\varphi}\langle \nu,\partial_r\rangle\big).
	\end{align*}
	Similarly, since $\langle f_*\overbar e_\nu,f_*\overbar e_\nu\rangle_M\leq\|df\|^2\cdot \langle \overbar e_\nu,\overbar e_\nu\rangle$, we have
	\begin{align*}
	\sum_\nu c_\partial(\overbar L e_\nu)^2=
	&\sum_{\lambda,\mu}\widehat A(f_*\overbar e_\lambda, f_*\overbar e_\mu)\cdot c_\partial(e_\lambda)c_\partial(e_\mu)\\
	=&-\sum_{\lambda} \widehat A(f_*\overbar e_\lambda, f_*\overbar e_\lambda) \\
	\geq & - \tr (\widehat A)=- H(g)+(n-1)f^*\big(\frac{\varphi'}{\varphi}\langle \nu,\partial_r\rangle\big).
	\end{align*}
	This finishes the proof.
\end{proof}

\subsection{Dihedral rigidity of radially convex domains in warped product manifolds}\label{subsec:2.4}

\

In this subsection, we prove the first main theorem of the paper (Theorem \ref{thm:rigidityOfWarpedProductintro}).
One of the key steps in the proof of Theorem \ref{thm:rigidityOfWarpedProductintro} is to find a non-zero solution of $\widehat D_\Psi$  in $H^1(N,S_N\otimes f^*S_M;B)$, where $\widehat D_\Psi$  is the modified Dirac operator from line \eqref{eq:Diracpotential}. According to  \cite[Theorem 6.1]{Wang:2021tq} and its proof, the operator $\widehat D_\Psi$ is essentially self-adjoint and Fredholm, provided the corresponding dihedral angles of $N$ and $M$ satisfy the \emph{strict} comparison inequality:
\[ \theta_{ij}(\overbar{g})_z <  \theta_{ij}(g)_{f(z)}\] for all $\overbar F_i, \overbar F_j$ and all $z \in \overbar F_i\cap \overbar{F}_j$.   However, Theorem \ref{thm:rigidityOfWarpedProductintro} only assumes the non-strict inequality on dihedral angles: 
\[ \theta_{ij}(\overbar{g})_z \leq  \theta_{ij}(g)_{f(z)}\] for all $\overbar F_i, \overbar F_j$ and all $z \in \overbar F_i\cap \overbar{F}_j$. So strictly speaking, \cite[Theorem 6.1]{Wang:2021tq} does not directly apply to the geometric setup of  Theorem \ref{thm:rigidityOfWarpedProductintro}. But as pointed out in the proof of   \cite[Theorem 1.7]{Wang:2021tq}, we can get around the above issue by approximating the non-strict comparison condition on dihedral angles by a sequence of strict comparison conditions on dihedral angles, which is achieved by ``artificially raising\footnote{More precisely, we introduce some auxiliary space dimensions and approximate the boundary condition $B$ of $\widehat D_\Psi$ (cf. Definition \ref{def:boundarycondition}) by a sequence of boundary conditions $B_n$ that satisfy the strict comparison condition on dihedral angles. See  the proof of \cite[Theorem 1.7]{Wang:2021tq} in \cite[Section 7]{Wang:2021tq} for the precise details.} the dihedral angles'' of $M$.  In particular, it follows from  \cite[Theorem 6.1]{Wang:2021tq} that 
the modified Dirac operator $\widehat D_\Psi$ with the boundary condition $B_n$ for each approximation is essentially self-adjoint and Fredholm. Moreover, the  domain of $(\widehat D_\Psi)_{B_n}$ is\footnote{In order to deal with general manifolds with \emph{polyhedral boundary}, we actually need to work with the Dirac operator associated to the map  $f\times \id\colon  N\times \mathbb I\to M\times \mathbb I$, where $\mathbb I$ is the unit interval and $\id\colon \mathbb I \to  \mathbb I$ is the identity map. Here the extra space $\mathbb I$ is introduced for the purpose of defining appropriate boundary conditions. But the potential $\Psi$ is still the same potential as that from line \eqref{eq:potential}. In other words, the potential $\Psi$ on $N\times \mathbb I$ is the pullback  of the potential $\Psi$ on $N$ via the projection map $N\times \mathbb I\to N$. The same remarks apply to the other constructions  in Subsection \ref{subsec:2.1}.} $H^1(N,S_N\otimes f^*S_M;B_n)$ and the Fredholm index  of $(\widehat D_\Psi)_{B_n}$ is equal to \[ \ind ((\widehat D_\Psi)_{B_n}) = \deg(f) \cdot \chi(M),\] where $\deg(f)$ is the degree of $f$
 and $\chi(M)$ is the Euler characteristic of $M$. By the assumption of Theorem \ref{thm:rigidityOfWarpedProductintro}, $ \deg(f) \cdot \chi(M)$ is nonzero, hence there exists nonzero element $\varphi_n\in H^1(N,S_N\otimes f^*S_M;B_n)$ such that   $(\widehat D_\Psi)_{B_n}(\varphi_n) = 0 $ for each $n$. Finally, we use this sequence $\{\varphi_n\}$ to obtain a nonzero element $\varphi \in H^1(N,S_N\otimes f^*S_M;B)$ such that   $(\widehat D_\Psi)_{B}(\varphi) = 0 $ by the following lemma, which is an analogue of \cite[Lemma 6.7]{Wang:2021tq}.

\begin{lemma}\label{lemma:solutionapproximate-poly}
	Let $f\colon N\to M$ be as in Theorem \ref{thm:rigidityOfWarpedProductintro} satisfying the same assumptions of Theorem \ref{thm:rigidityOfWarpedProductintro}.
	For each codimension one face $\overbar F_i$ of $N$ and its corresponding face  $F_i$ of $M$, let  $u_i$ and $v_i$ be their unit inner normal vector field, respectively. Let $\{v_{i,t}\}_{t\in[0,1]}$ be a path of smooth vector fields on $F_i$ such that $v_{i,0}=v_i$, and $v_{i,t}\to v_i$ in $C^1$-norm as $t\to 0$. Let $B_t$ be the boundary condition on sections of $S_N\otimes f^*S_M$ that on each face $\overbar F_i$ is given by
	$$(\overbar \epsilon\otimes\epsilon)(\overbar c(u_i)\otimes c(v_{i,t}))\sigma=-\sigma,$$
	where $\overbar \epsilon$ and $\epsilon$ are the $\Z_2$-grading operators on $S_N$ and $f^*S_M$. If there exists a nonzero $\sigma_t\in H^1(N,S_N\otimes f^*S_M;B_t)$ such that $\widehat D_\Psi\sigma_t=0$ for each $t>0$, then there exists a nonzero $\sigma\in H^1(N,S_N\otimes f^*S_M;B_0)$ such that  $\widehat D_\Psi\sigma=0$ and $\mathcal P\sigma=0$, where $\mathcal P$ is the Penrose operator from line \eqref{eq:penrosedef}. 
\end{lemma}
\begin{proof}
	We may assume without loss of generality that  that $\|\sigma_t\|=1$.
	By Proposition \ref{prop:comparison}, for each $t>0$ and $\sigma_t\in H^1(N,S_N\otimes f^*S_M;B_t)$, we have
	\begin{equation}\label{eq:comparison2}
			0=\int_N|\tdirac_\Psi\sigma_t|^2\geq \int_N |\mathcal P\sigma_t|^2 + \int_N\mathcal R|\sigma_t|^2+\int_{\partial N}\mathcal A_t|\sigma_t|^2,
	\end{equation}
	where $\mathcal R$ and $\mathcal A_t$ are bounded endomorphisms of $S_N\otimes f^*S_M$. By assumption, we have $\mathcal A_t\to \mathcal A_0$ as $t\to 0$, since $\mathcal A_t$ determined by $v_{i, t}$ and the derivatives of $v_{i, t}$. See the proof of Proposition \ref{prop:comparison} for more details. 
	
	By the assumption of Theorem \ref{thm:rigidityOfWarpedProductintro},   we have $\mathcal R\geq 0$ and $\mathcal A_0\geq 0$. It follows from  Proposition \ref{prop:comparison} that 
	$$\lim_{t\to 0}\|\mathcal P\sigma_t\|=0.$$
	By the definition of $\mathcal P$ and $\widehat D_\Psi$ (cf. line \eqref{eq:penrose} and line \eqref{eq:Diracpotential}), we see that the $H^1$-norm of $\sigma\in H^1(N,S_N\otimes f^*S_M)$ is equivalent to the following norm
	$$(\|\sigma\|^2+\|\mathcal P\sigma\|^2+\|\widehat D_\Psi\sigma\|^2)^{1/2}.$$
	Therefore, the $H^1$-norms of $\{\sigma_t\}_{t\in[0,1]}$ are uniformly bounded. By the Rellich Lemma, there exists a convergent sequence $\{\sigma_{t_n}\}$ with respect to the $L^2$-norm, where $t_n\to 0$ as $n\to\infty$. Assume that $\sigma_{t_n}\to\sigma$ in $L^2(N,S_N\otimes f^*S_M)$. In particular, we have $\|\sigma\|=1$. 
	
	Since $\|\mathcal P\sigma_{t_n}\|\to 0$ and $\widehat D_\Psi\sigma_{t_n}=0$, it follows that  $\{\sigma_{t_n}\}$ is also a Cauchy sequence with respect to the $H^1$-norm. Therefore $\sigma$ lies in $ H^1(N,S_N\otimes f^*S_M)$ and $\sigma_{t_n}\to\sigma$ with respect to the $H^1$-norm. Note that  $\mathcal P$ and  $\widehat D_\Psi$ are bounded as linear maps from $ H^1(N,S_N\otimes f^*S_M)$ to $L^2(N,S_N\otimes f^*S_M)$,  and the Sobolev trace map is a bound linear map from $ H^1(N,S_N\otimes f^*S_M)$ to $H^{1/2}(\partial N,S_N\otimes f^*S_M)$.  Consequently, we  conclude that $\sigma$ lies in $H^1(N,S_N\otimes f^*S_M;B_0)$,  $\mathcal P\sigma=0$ and $\widehat D_\Psi\sigma=0$. This finishes the proof.  
\end{proof}

Now let us prove  Theorem \ref{thm:rigidityOfWarpedProductintro}. 

\begin{proof}[{Proof of Theorem \ref{thm:rigidityOfWarpedProductintro}}]
	Without loss of generality, we assume that both $N$ and $M$ are even-dimensional. The computation for odd-dimensional case is similar and we omit the proof. Alternatively, one can reduce the odd dimensional case to an even dimensional case by considering direct products $N\times I$ and $M\times I$. 
	
	Let $H^1(N,S_N\otimes f^*S_M;B)$ be the Sobolev $H^1$-space of sections of $S_N\otimes f^*S_M$ over $N$ that satisfy the boundary condition $B$ in Definition \ref{def:boundarycondition}. Note that $\widehat D_\Psi$ only differs from the usual twisted Dirac operator on $S_N\otimes f^*S_M$ by a bounded endomorphism. It follows from  Lemma \ref{lemma:solutionapproximate-poly} and the discussion before Lemma \ref{lemma:solutionapproximate-poly} that there exists a non-zero element  $\sigma$  in  $H^1(N,S_N\otimes f^*S_M;B)$ such that $\tdirac_\Psi\sigma=0$ and $\mathcal P\sigma=0$.

By definition of the Penrose operator in line \eqref{eq:penrosedef},  for any tangent vector $\overbar \xi$ in $TN$, we have
\begin{equation*}
\tcon_{\xi}\sigma+\frac 1 n\overbar c(\overbar \xi)\tdirac\sigma = \mathcal P_{\overbar \xi}\sigma=0.
\end{equation*}
Since we also have $(\tdirac + \Psi)
\sigma = \tdirac_\Psi \sigma=0$, it follows that 
$$\tdirac\sigma=-\Psi\sigma.$$
Therefore, the two equations above together imply that 
\begin{equation}\label{eq:ODEsigma}
	 \tcon_{\xi}\sigma - \frac 1 n\overbar c(\overbar \xi)\Psi\sigma = 0 
\end{equation}
for all $\overbar \xi\in TN$. 

As a consequence of line \eqref{eq:ODEsigma},  $\sigma$ is smooth on $N$. Moreover,  $\sigma$ extends continuously everywhere over $N$ in the following sense. For each $x\in N$ on a face with codimension $\geq 2$, consider a local chart near $x$ so that sections of $S_N\otimes f^*S_M$ near $x$ is identified with a vector in $\R^{2^n}$. Let $\gamma$ be a path on $N$ that $\gamma(t)$ lies in the interior of $N$ except $\gamma(0)=x$. Then the equation in line \eqref{eq:ODEsigma} along $\gamma$ gives rise to ordinary differential equations in $t$, whose coefficients are uniformly bounded by assumption. In particular, since $\sigma$ satisfies line \eqref{eq:ODEsigma}, it satisfies the differential equations. Therefore $\sigma$ restricted on $\gamma$, as a function with value in $\R^{2^n}$, is Lipschitz continuous for $t>0$. Hence $\sigma$ along $\gamma$ extends continuously to $\gamma(0)=x$, which we will call $\sigma(x)$. The value $\sigma(x)$ is independent of the choice of $\gamma$.

 For a given point $x\in N$ and any smooth path in $N$ starting at $x$, the restriction of $\sigma$ on the path  satisfies the homogeneous differential equation given by  \eqref{eq:ODEsigma} along the path. By the uniqueness of the solution, if $\sigma(x)=0$, then $\sigma$ vanishes on the whole path. Without loss of generality, we may assume  $N$ is connected. The above discussion implies that if $\sigma$ vanishes at any point in $N$, then  $\sigma$ has to vanish on the whole $N$, which would contradict the fact that $\sigma$ is a nontrivial section.  Therefore $\sigma$ is non-zero everywhere on $N$. Now since $\tdirac_\Psi \sigma= 0$,  the inequality \eqref{eq:comparison} of Proposition \ref{prop:comparison} implies that \begin{enumerate}[label=$(\roman*)$]
	\item 
	$\Sc(\overbar{g})_x = \Sc(g)_{f(x)}$ for all $x\in N$, 
	\item $H_{\overbar{g}}(\overbar{F}_i)_y =  H_{g}(F_i)_{f(y)}$ for all $y$ in each codimension one face $\overbar{F}_i$ of $N$. 
\end{enumerate} 

Part $(iii)$ and $(iv)$ follow from the same argument as that in the proof of \cite[Theorem 2.9]{Wang2022:fl}. For the convenience of the reader, we repeat the argument here. Let $\overbar x$ is a point in the intersection of $m$ codimension one faces of $N$ whose unit inner normal vectors at $\overbar x$ are denoted by $\overbar \nu_1,\ldots,\overbar \nu_m$. Let $\nu_1, \ldots, \nu_m$ be the unit inner normal vectors at $f(\overbar x)$ of the corresponding faces of $M$. Since $\sigma$ extends continuously at $\overbar x$ and satisfies the $m$ boundary conditions, we have
$$(\overbar \epsilon\overbar c(\overbar \nu_j)\otimes 1)\sigma(\overbar x)=-(1\otimes \epsilon c(\nu_j))\sigma(\overbar x),~\forall j=1,\ldots m.$$
For $a=(a_1,a_2,\ldots,a_m)\in\R^m$, we define
$$\nu_a\coloneqq \sum_{j=1}^m a_j\nu_j \textup{ and } \overbar \nu_a\coloneqq \sum_{j=1}^m a_j\overbar \nu_j.$$
Clearly, we have 
$$(\overbar \epsilon\overbar c(\overbar \nu_a)\otimes 1)\sigma(\overbar x)=-(1\otimes \epsilon c(\nu_a))\sigma(\overbar x).$$
By taking vector norms of both sides, we obtain
$$|\overbar \nu_a|_{\overbar g}^2\cdot |\sigma(\overbar x)|^2=|\nu_a|_g^2\cdot |\sigma(\overbar x)|^2, $$
since $c(\nu_a)(-c(\nu_a)^*)=c(\nu_a)^2=|\nu_a|_{g}^2$ and $\overbar c(\overbar \nu_a)(-\overbar c(\overbar \nu_a)^*)=\overbar c(\overbar \nu_a)^2=|\overbar \nu_a|_{\overbar g}^2$. It follows that 
$|\overbar \nu_a|_{\overbar g}=|\nu_a|_g$, since $|\sigma(\overbar  x)|\ne 0$. 

Consider the two symmetric quadratic forms
$$\overbar Q,Q\colon \R^m\times\R^m\to\R$$
defined by
$$\overbar Q(a,b)\coloneqq \langle \overbar \nu_a,\overbar \nu_b\rangle_{\overbar g} \textup{ and } Q(a,b)\coloneqq \langle \nu_a,\nu_b\rangle_{g}.$$
The above discussion shows that 
$$Q(a,a)=\overbar Q(a,a),~\forall a\in\R^m.$$
By the polarization identity, we have
$$Q(a,b)=\frac 1 4(Q(a+b,a+b)-Q(a-b,a-b)),$$
$$\overbar Q(a,b)=\frac 1 4(\overbar Q(a+b,a+b)-\overbar Q(a-b,a-b)).$$
Hence $Q$ and $\overbar Q$ are identical. In particular, we have 
$$\langle \nu_i,\nu_j\rangle_g=\langle \overbar \nu_i,\overbar \nu_j\rangle_{\overbar g},~\forall i,j=1,\ldots,m.$$
This proves part $(iv)$. Now part $(iii)$ is a direct consequence of part $(iv)$.

Now we assume that $\varphi$ is strictly log-concave. 
Since the inequality  \eqref{eq:comparison} achieves equality at $\sigma$, the proof of Proposition \ref{prop:comparison} shows that the inequality \eqref{eq:gradpsi} also becomes  equality at $\sigma$, that is,
\begin{equation}\label{eq:dr=1}
(\overbar\epsilon\otimes\epsilon)(\overbar c(\grad\overbar r))\otimes c(\partial_r)\sigma=-\sigma, 
\end{equation}
where $\overbar r = f^\ast r$ and $\grad \overbar r$ is the gradient of $\overbar r$. 

Let us write  $\alpha=\overbar\epsilon\otimes\epsilon c(\partial_r)$, then Equation \eqref{eq:dr=1} becomes 
\begin{equation}\label{eq:dr=1equiv} \alpha\overbar c(\grad\overbar r)\sigma=-\sigma,  \textup{ or equivalently, }   \overbar c(\grad\overbar r)\sigma=-\alpha\sigma.
\end{equation}
Since $\sigma$ is non-zero everywhere, we have that $|\grad\overbar r|=1$ everywhere on $N$. 

Now we show that the metric $\overbar g$ on $N$ is also a warped product metric of the form $d\overbar r^2+\overbar \varphi^2\overbar h$ with $f_*(\partial_{\overbar r})=\partial_r$ and $\overbar\varphi=f^*\varphi$.
In the following computation, strictly speaking, we should use $f^\ast\varphi$ and $f^\ast(\varphi')$ in place of  $\varphi$ and $\varphi'$, where  $f^\ast\varphi$ and $f^\ast(\varphi')$  are the pullbacks of the functions $\varphi$ and $\varphi'$ from $M$ on $N$. But in order to avoid cumbersomeness of our notation, we will continue to  write $\varphi$ and $\varphi'$, since no confusion is likely to arise. 

By definition of $\Psi$ in line \eqref{eq:potential} and line \eqref{eq:dr=1equiv}, we have 
\begin{equation}\label{eq:nablasigma}
\tcon_{\overbar\xi}\sigma=\frac 1 n\overbar c(\overbar \xi)\Psi \sigma=
\frac{\varphi'}{2\varphi}\overbar c(\overbar \xi)\alpha\sigma=-\frac{\varphi'}{2\varphi}\overbar c(\overbar \xi)\overbar c(\grad\overbar r)\sigma.
\end{equation}
 Let us first  compute the Hessian of the function $\overbar r = f^\ast r$ over $N$. We have 
\begin{align*}
\overbar c(\ncon_{\overbar \xi} \grad\overbar r)\sigma=&\nabla_{\overbar\xi}( \overbar c(\grad\overbar r) \sigma)- \overbar c(\grad\overbar r)\nabla_{\overbar\xi}\sigma\\
=&\tcon_{\overbar\xi}( \overbar c(\grad\overbar r) \sigma)- \overbar c(\grad\overbar r)\tcon_{\overbar\xi}\sigma  \quad \textup{ by definition of $\tcon$ in line \eqref{eq:hatcon} }\\
=& -(\alpha+\overbar c(\grad\overbar r))\tcon_{\overbar\xi}\sigma \quad \textup{  by line \eqref{eq:dr=1equiv} }\\
=&\frac{\varphi'}{2\varphi}(\alpha+\overbar c(\grad\overbar r)) \overbar c(\overbar \xi)\overbar c(\grad\overbar r) \sigma \quad \textup{  by line  \eqref{eq:nablasigma}} \\
=&\frac{\varphi'}{2\varphi}\overbar c(\overbar\xi)\sigma +\frac{\varphi'}{2\varphi}\overbar c(\grad\overbar r)\overbar c(\overbar\xi)\overbar c(\grad\overbar r)\sigma \quad \textup{  by line \eqref{eq:dr=1equiv} } \\
=&2\frac{\varphi'}{2\varphi}\left(\overbar c(\overbar \xi)-\langle\overbar \xi,\grad\overbar r\rangle \overbar c(\grad\overbar r)\right)\sigma
= \frac{\varphi'}{\varphi}\overbar c\big(\overbar \xi-\langle\overbar \xi,\grad\overbar r\rangle\grad\overbar r\big)\sigma.
\end{align*} 
It follows that 
\begin{equation*}
\overbar c\left(\ncon_{\overbar \xi} \grad\overbar r-\frac{\varphi'}{\varphi}\left(\overbar \xi-\langle\overbar \xi,\grad\overbar r\rangle\grad\overbar r\right)\right)\sigma=0.
\end{equation*}
Since $\sigma$ is non-zero everywhere, we obtain that
\begin{equation*}
\ncon_{\overbar \xi} \grad\overbar r=\frac{\varphi'}{\varphi}\left(\overbar \xi-\langle\overbar \xi,\grad\overbar r\rangle\grad\overbar r\right).
\end{equation*}
Equivalently, we have
\begin{equation*}
\hess\overbar r=\frac{\varphi'}{\varphi}\big(\overbar g-\grad\overbar r\otimes \grad\overbar r\big).
\end{equation*}
In particular, the flow lines generated by  the gradient $\grad\overbar r$ of $\overbar r$ on $N$ are unit speed geodesics. 


Let $x$ be a point in $M$. Along the geodesic generated by $\grad{\overbar r}$ near $x$,  the metric $\overbar g$ of $N$ decomposes into 
$$\overbar g= (\grad{\overbar r})^2+\overbar g_{\overbar r},$$
where $\overbar g_{\overbar r}$ is a metric on the level set of $\overbar r$.
By definition of the Hessian (cf. \cite[Proposition 7, page 47]{Petersen}), we have 
$$L_{\partial_{\overbar r}}(\overbar g)=2\hess\overbar r=\frac{2\varphi'}{\varphi}\big(\overbar g-\grad\overbar r\otimes \grad\overbar r\big),$$
where $L_{\partial_{\overbar r}}(\overbar g)$ is the Lie derivative of $\overbar g$ with respect to $\partial_{\overbar r}$. 
Hence for $\overbar g_{\overbar r}$ from the above decomposition of $\overbar g$, we have 
\begin{equation*}
L_{\partial_{\overbar r}}(\overbar g_{\overbar r})=\frac{2\varphi'}{\varphi}\overbar g_{\overbar r}.
\end{equation*}
Recall that we have written $\varphi$ to mean $f^\ast \varphi$ in order to avoid cumbersomeness of our notation. So the above computation shows that  $\overbar g_{\overbar r}=(f^\ast\varphi)^2\overbar h$ for some fixed metric $\overbar h$ on every level set of the function $\overbar r$. Therefore $N$ is a codimension zero submanifold of the warped product manifold $(\overbar X\times J,\overbar g=(\grad \overbar r)^2+(f^\ast\varphi)^2\overbar h)$. 

Since we have shown that the metric $\overbar g$ is also a warped product metric $\overbar g_{\overbar r}=(f^\ast\varphi)^2\overbar h$, the scalar curvature formula for warped product metrics (cf. Equation \eqref{eq:warpscalar}) and  the inequality in Lemma \ref{lemma:curvature>=} (which is now an equality by the above discussion) imply that $\Sc(\overbar h) = f^\ast \Sc(h)$, and the same proof for Part (I) of \cite[Theorem 1.7]{Wang:2021tq} shows that $f\colon (N, \overbar g) \to (M, g)$ is a local isometry. This proves part (I). 

 Part  (II) follows from a straightforward computation similar to that in the proof of  Lemma \ref{lemma:secondff>=} (cf. the proof of \cite[Proposition 3.1]{Wang2022:fl} for the precise details). This finishes the proof. 
\end{proof}

\section{Dihedral rigidity of  hyperbolic manifolds with polyhedral boundary}\label{sec:hyperbolic}

In this section, by combining the methods developed by the authors in \cite{Wang2022:fl}, we extend the techniques in Section \ref{sec:radconvexrigidity} to prove a dihedral rigidity theorem  for a large class of hyperbolic manifolds with polyhedral boundary.  

More precisely, we view the standard hyperbolic space $\mathbb H^n$ as a warped product manifold as follows. Let $X$ be the standard Euclidean space $\R^{n-1}$ and $h$  the flat metric  on $\R^{n-1}$. Then $\mathbb H^n$ is isometric to the warped product space  $\R^n=X\times \R$ with the metric
$$g=dr^2+e^{2r}h.$$ 
Recall that each  leaf $\R^{n-1}\times \{r\}$ is usually called a horosphere.

Now let us prove Theorem \ref{thm:hyperbolicComparison-intro}
\begin{proof}[Proof of Theorem \ref{thm:hyperbolicComparison-intro}]
	Without loss of generality, we assume that both $N$ and $M$ are even-dimensional. The computation for odd-dimensional case is similar and we omit the proof.
	
		Let $H^1(N,S_N\otimes f^*S_M;B)$ be the Sobolev $H^1$-space of sections of $S_N\otimes f^*S_M$ over $N$ that satisfy the boundary condition $B$ in Definition \ref{def:boundarycondition}. Similar to the proof of Theorem \ref{thm:rigidityOfWarpedProductintro}, there exists a nonzero element  $\sigma\in H^1(N,S_N\otimes f^*S_M;B)$ in the kernel  of  $\tdirac_\Psi=\tdirac+\Psi$. 
	Since the metric $h$ on $X = \mathbb R^{n-1}$ is flat and $(\log\varphi)''=0$, Lemma \ref{lemma:curvature>=} and line \eqref{eq:gradpsi} holds automatically without assuming $f$ is distance-non-increasing in the interior. Only Lemma \ref{lemma:secondff>=} requires the assumption that  $f$ is distance-non-increasing on the boundary. It follows that Proposition \ref{prop:comparison} still applies in this case. 
	The same argument from the proof of Theorem \ref{thm:rigidityOfWarpedProductintro} shows that $\sigma$ is non-zero everywhere. Furthermore, part $(i), (ii), (iii)$ and $(iv)$  follow from the same proof of part $(i), (ii), (iii)$ and $(iv)$ of Theorem \ref{thm:rigidityOfWarpedProductintro}.


Now let us prove part $(v)$. For brevity, let us write  $\alpha=(\overbar \epsilon\otimes\epsilon)(1\otimes c(\partial_r))$.  Since $\varphi = e^{r}$ in the current setup, we have 
	\[ \Psi =  \frac{n}{2}f^*(\frac{\varphi'}{\varphi})\cdot (\overbar\epsilon\otimes\epsilon)\cdot (1
	\otimes  c(\partial_r)) = \frac{n}{2} \alpha.\]
	Since $\tdirac\sigma=-\Psi\sigma$, it follows that 
	\begin{equation}\label{eq:nablaHatXiSigma}
	\tcon_{\overbar \xi}\sigma=\frac 1 2\overbar c(\overbar \xi)\alpha\sigma,
	\end{equation}
or equivalently, \begin{equation}\label{eq:nablaXiSigma}
	\nabla_{\overbar\xi}\sigma=\frac 1 2\overbar c(\overbar \xi)\alpha\sigma-\frac 1 2 c(f_*\overbar\xi\wedge\partial_r)\sigma
\end{equation}
for all $\overbar \xi\in TN$.

	Let $F^t$ be the top of $M$, namely $F^t$ is contained in a leaf $X\times\{r_0\}$, and every point in $M$ connects to a point in $F^t$ via a geodesic in the direction of $\partial_r$. Let $\overbar F^t$ be the face of $N$ that is mapped by $f$ to the top face $F^t$ of $M$. Now we show that the top face $\overbar F^t$ of $N$ is flat. On $\overbar F^t$, the section $\sigma$ satisfies the boundary condition 
	$$(\overbar\epsilon\otimes\epsilon)(\overbar c(\overbar e_n)\otimes c(e_n))\sigma=-\sigma$$
	where $\overbar\epsilon\otimes\epsilon$ is the grading operator on $S_N\otimes f^*S_M$, $\overbar e_n$ is the unit inner normal vector of $\overbar F^t$, and $e_n=-\partial_r$ is the unit inner normal vector of $F^t$. Equivalently, $\overbar c(\overbar e_n)\sigma=\alpha\sigma$, where $\alpha = (\overbar \epsilon\otimes\epsilon)(1\otimes c(\partial_r))$. In particular, Equation \eqref{eq:nablaXiSigma} becomes
	$$\nabla_{\overbar\xi}\sigma= -\frac 1 2 \overbar c(\overbar e_n)\overbar c(\overbar \xi)\sigma-\frac 1 2 c(e_n)c(f_*\overbar \xi)\sigma$$
	for all  $\overbar\xi \in T\overbar F^t$. Moreover,  note that 
	\begin{align*}
	\overbar c(\ncon_{\overbar\xi}\overbar e_n)\sigma=&\nabla_{\overbar\xi}\big(\overbar c(\overbar e_n)\sigma\big)-\overbar c(\overbar e_n)\nabla_{\overbar\xi}\sigma=\tcon_{\overbar\xi}\big(\overbar c(\overbar e_n)\sigma\big)-\overbar c(\overbar e_n)\tcon_{\overbar\xi}\sigma\\
	=&\alpha\tcon_{\overbar\xi}\sigma-\overbar c(\overbar e_n)\tcon_{\overbar\xi}\sigma
	=\alpha\frac 1 2\overbar c(\overbar\xi)\alpha\sigma-\overbar c(e_n)\frac 1 2\overbar c(\overbar\xi)\alpha\sigma\\
	=&-\overbar c(\overbar\xi)\sigma 
	\end{align*}
for all  $\overbar\xi \in T\overbar F^t$.
	It follows that  the second fundamental form $A$ of $\overbar F^t$ in $N$ is given by $A(\overbar\xi)=-\overbar\xi$.	
To summarize, we see that 
	\begin{equation}
	\nabla^\partial_{\overbar\xi}\sigma=0,
	\end{equation}
for all  $\overbar\xi \in T\overbar F^t$, where $\nabla^\partial$ is the boundary connection defined in line \eqref{eq:boundaryconnection}.  In other words, the restriction of $\sigma$ on $\overbar F^t$ is parallel with respect to $\nabla^\partial$.   Note that the connection $\nabla^\partial$ is the precisely the canonical Riemannian connection on $S_N\otimes f^\ast S_M$ over $\overbar F^t$ with the respect to the metric $\overbar g|_{\overbar F^t}$ on $\overbar F^t$ and the metric  $h$ on $F^t$. Since the top face $F^t$ of $M$ is flat domain in $\mathbb R^{n-1}$, there exists a set of smooth sections $\{e_1, \cdots, e_n\}$ of $S_M|_{F^t}$ such that each $e_i$ is parallel with respect to the connection $\nabla^{S_M,\partial}$ (from line  \eqref{eq:nablaSMBoundary}) and  $\{e_1, \cdots, e_n\}$ forms an orthonormal basis at every point $x\in F^t$. Now by applying each of the Clifford multiplications $\{c(e_{i_1}\wedge \cdots\wedge e_{i_k})\}$   to  $\sigma$ and following the same argument in the proof of \cite[Theorem 2.9]{Wang2022:fl}, we see that the bundle $S_N|_{\overbar F^t}$ is flat with respect to the connection $\nabla^{S_N,\partial}$ (from line \eqref{eq:nablaSNBoundary}). It follows that the metric $\overbar g|_{\overbar F^t}$ on $\overbar F^t$ is flat.
	
	Now we prove that $N$ is also hyperbolic. Let $v_1,\ldots,v_{n-1}$ be the vector fields on $M$ generated by the coordinates of $X\subset \R^{n-1}$. In particular, every $v_i$ is parallel along each leaf of $M$.   A direct computation shows that
	$$\mcon_{v_i}v_j=-\delta_{ij}\partial_r,~\mcon_{\partial_r}v_j=v_j.$$
	Set $w_j=e^{-r}v_j$. For any smooth section $\eta$ of $S_M$, we have
	\begin{align*}
	&\tcon^{S_M}_{v_i}\big( c(w_j)\eta\big)-c(w_j)\tcon^{S_M}_{v_i}\eta\\
	=&c(\mcon_{v_i}w_j)\eta+\frac 1 2 c(v_i)c(\partial_r)c(w_j)\eta-\frac 1 2c(w_j)c(v_i)c(\partial_r)\eta\\
	=&-e^{-r}\delta_{ij}c(\partial_r)\eta+\langle v_i,w_j\rangle c(\partial_r)\eta=0
	\end{align*}
	and
	\begin{align*}
	&\tcon^{S_M}_{\partial_r}\big( c(w_j)\eta\big)-c(w_j)\tcon^{S_M}_{\partial_r}\eta\\
	=&c(\mcon_{\partial_r}w_j)\eta=
	c\big(\mcon_{\partial_r}(e^{-r}v_j)\big)\eta=0.
	\end{align*}
In particular,  for the chosen section $\sigma$ above, we have
	\begin{equation}\label{eq:twistedparallel}
	\tcon_{\overbar\xi}\big( c(w_j)\sigma\big)=c(w_j)\tcon_{\overbar \xi}\sigma.
	\end{equation}
	all $\overbar\xi \in TN$. 
	
	Let us set 
	$$\Lambda=\{w_1,\ldots,w_{n-1},\partial_r \}.$$
	For any subset $\lambda\subseteq\Lambda$, we define
	$$\sigma_{\lambda}\coloneqq \Big(\prod_{w\in\lambda}c(w)\Big)\sigma.$$
It follows from  line \eqref{eq:nablaHatXiSigma}  and line \eqref{eq:twistedparallel} that 
	\begin{equation}\label{eq:nablaXiSigmaLambda}
	\tcon_{\overbar \xi}\sigma_\lambda=\pm\frac 1 2\overbar c(\overbar \xi)\alpha\sigma_\lambda.
	\end{equation}
	
	Since we have shown $(\overbar F^t, \overbar g|_{\overbar F^t})$ is flat, a direct computation shows that the collection $\{\sigma_\lambda|_{\overbar F_0}\}_{\lambda\in\Lambda}$ is linearly independent at every point $x\in \overbar F^t$ (cf. the proof of \cite[Theorem 2.9]{Wang2022:fl}). In particular, there exists an open neighborhood $U$ of $\overbar F^t$ such that  $\{\sigma_\lambda\}_{\lambda\in\Lambda}$ is linearly independent at every point in $U$. Denote   the curvature form associated to the connection $\tcon$  on $S_N\otimes f^*S_M$ by
	$$\widehat{\mathcal R}_{v,w}\coloneqq [\tcon_v,\tcon_w]-\tcon_{[v,w]}.$$
	From line \eqref{eq:nablaXiSigmaLambda}, we see that
	$$\widehat{\mathcal R}_{v,w}(\sigma_\lambda)=\frac 1 2\overbar c(v\wedge w)\sigma_\lambda.$$
	Since the connection $\tcon^{S_M}$ on $S_M$ is flat, it follows  that the curvature form  $\mathcal R^{S_N}$ for the connection $\nabla^{S_N}$ on $S^N$ is given by 
	\[  {\mathcal R}^{S_N}_{v,w}(\eta)=\frac 1 2\overbar c(v\wedge w)\eta \]
	for all smooth sections $\eta$ of $S_N$.  By \cite[Theorem 2.7]{spinorialapproach}, we see that the Levi-Civita connection $\ncon$ on $TN$ over $U$ has constant sectional curvature $-1$, hence the metric $\overbar g$ on $U$ is hyperbolic. Moreover, since the top face $\overbar F^t$ is flat and has mean curvature $(n-1)$, it follows that $\overbar F^t$ is locally a horosphere.
	
Note that $\{\sigma_\lambda\}$ being linearly independent is an open condition. The above discussion showed that if $\{\sigma_\lambda\}$  is linearly independent at $x\in N$, then the metric $\overbar g$ is hyperbolic in a neighborhood of $x$. To show that $\overbar g$ is hyperbolic on the whole $N$,  it suffices to show that the set where $\{\sigma_\lambda\}$ is linearly independent is also closed, since $N$ is connected. For any $y \in N$, consider a smooth path $\gamma$ in $N$ connecting $y$ to a point $x_0$ in $\overbar F^t$. Without loss of generality, we assume $\{\sigma_\lambda\}$ is linearly independent at every $x\in \gamma$ except possibly at $y$. By the above discussion, there exists a neighborhood $V$ of $\gamma\backslash\{y\}$ in $N$ such that the metric $\overbar g$ have constant sectional curvature $-1$ on $V$. By continuity, we see the metric $\overbar g$ has constant sectional curvature $-1$ on the closure $\overbar V$ of $V$. By lifting all data to the universal covering of $\overbar V$ if necessary, we may  without loss of generality view $\overbar V$ as a subspace of the hyperbolic space $\mathbb H^n$.  Since $M$ is a subspace of $\mathbb H^n$ by assumption, it follows that the pullback  bundle $f^\ast TM$ can be isometrically identified with the bundle $TN$ over $\overbar V$, which induces an identification between $f^\ast S_M$ and $S_N$ over $\overbar V$.  We emphasize that such an identification is generally \emph{not} induced by the map $f$. Without loss of generality, we can choose the above identification so that it identifies the unit inner normal vector $\overbar e_n$ with $-\partial_r$. Let us denote by $\partial_{\overbar r}$ the vector in $T_xN$ corresponding to the vector $\partial_r \in (f^\ast TM)_x$ for each $x\in \overbar V$ under the above identification. Furthermore, if we view the function $r$ on $M$ as a scalar endomorphism of $TM$, then we denote by $\overbar r$ the corresponding endomorphism on $TN\cong f^\ast TM$ over $\overbar V$. 

Let us define a new connection 
\begin{equation}\label{eq:flatconn}
\widetilde \nabla \coloneqq \tcon^{S_N}\otimes 1 + 1\otimes  \tcon^{S_M}
\end{equation} 
where $\tcon^{S_N}_{\overbar \xi} \coloneqq  \nabla_{\overbar \xi}^{S_N} + \frac{1}{2} \overbar c(\prescript{N}{}\nabla_{\overbar \xi} \partial_{\overbar r}) \overbar c(\partial_{\overbar r}) $ and $\tcon^{S_M}$ is given in line \eqref{eq:nablaHat}. It follows from  Equation \eqref{eq:nablaHatXiSigma} that 
\begin{equation}\label{eq:parallel}
	\widetilde \nabla_{\overbar \xi}\sigma =  \frac 1 2\overbar c(\overbar \xi)\alpha\sigma  + \frac{1}{2} \overbar c(\prescript{N}{}\nabla_{\overbar \xi} \partial_{\overbar r}) \overbar c(\partial_{\overbar r}) \sigma
\end{equation}
where $\alpha=(\overbar \epsilon\otimes\epsilon)(1\otimes c(\partial_r))$. Note that we have 
\[ \tcon_{\overbar \xi}^{S_N} \overbar c(\partial_{\overbar r}) = \overbar c(\partial_{\overbar r}) \tcon_{\overbar \xi }^{S_N} \textup{ and }  \tcon_{\xi}^{S_M} c(\partial_r) = c(\partial_r) \tcon_{\xi}^{S_M}  \]
for all $\overbar \xi\in TN$ and $\xi \in TM$. A direct computation shows that 
\begin{align*}
  \widetilde \nabla_{\overbar \xi} \big ( (\overbar\epsilon\otimes\epsilon)(\overbar c(\partial_{\overbar r})\otimes c(\partial_r)) e^{-\overbar r/2}\sigma \big)
 = &  (\overbar\epsilon\otimes\epsilon)(\overbar c(\partial_{\overbar r})\otimes c(\partial_r)) \widetilde \nabla_{\overbar \xi}(e^{-\overbar r/2}\sigma)  \\
 = &  - \widetilde \nabla_{\overbar \xi}(e^{-\overbar r/2}\sigma)
\end{align*} 
for all $\overbar \xi \in TN|_{\overbar V}$. It follows that the section 
\[  \rho \coloneqq (\overbar\epsilon\otimes\epsilon)(\overbar c(\partial_{\overbar r})\otimes c(\partial_r)) e^{-\overbar r/2}\sigma + e^{-\overbar r/2}\sigma \]
is parallel with respect to $ \widetilde \nabla$.  

Recall that  
$\sigma$ satisfies the boundary condition $B$, that is,  
$$(\overbar\epsilon\otimes\epsilon)(\overbar c(\overbar e_n)\otimes c(e_n))\sigma=-\sigma \textup{ at $x_0 \in \overbar F^t$,} $$
or equivalently, 
$$(\overbar\epsilon\otimes\epsilon)(\overbar c(\partial_{\overbar r})\otimes c(\partial_r)) \sigma= - \sigma \textup{ at $x_0 \in \overbar F^t$,}$$ since  $\overbar e_n = - \partial_{\overbar r}$ and $e_n = -\partial_r$ at $x_0$. Therefore, $\rho = 0  \textup{ at } x_0 \in \overbar F^t.$
It follows that $\rho \equiv 0$ on $\overbar V$, since $\rho$ is parallel with respect to $\widetilde \nabla$. In other words, we have 
\[ (\overbar\epsilon\otimes\epsilon)(\overbar c(\partial_{\overbar r})\otimes c(\partial_r)) \sigma= - \sigma  \textup{ over $\overbar V$}.\]
By plugging the above equality into Equation \eqref{eq:parallel}, we see that 
\[ \widetilde \nabla_{\overbar \xi}(e^{-\overbar r/2}\sigma) = 0  \]
for all  $\overbar \xi \in TN|_{\overbar V}$. Similarly, it follows from  Equation \eqref{eq:twistedparallel} that each $e^{-\overbar r/2}\sigma_\lambda$ is parallel with respect to $\widetilde \nabla$. Consequently, we have 
\[ e^{-\overbar r(x)}\langle \sigma_\lambda, \sigma_\mu\rangle_x = e^{-\overbar r(x_0)}  \langle \sigma_\lambda, \sigma_\mu\rangle_{x_0}\]
for all $x\in \overbar V$ and all subsets $\lambda, \mu$ of $\Lambda$. Recall that $\{\sigma_\lambda\}$ is linearly independent at $x_0$. Hence  $\{\sigma_\lambda\}$ is linearly independent at every $x\in \overbar V$. 
 This completes the proof that the metric $\overbar g$ is hyperbolic on the whole $N$.

	 Part (I) follows from the same argument for part  (II) of Theorem \ref{thm:rigidityOfWarpedProductintro}. Let us now prove part (II).    By lifting all data to the universal covering of $N$ if necessary, we can assume without loss of generality view $N$ as a subspace of the hyperbolic space $\mathbb H^n$ and $\overbar F^t$  part of a horosphere in $\mathbb H^n$.   Since $M$ is a subspace of $\mathbb H^n$ by assumption, it follows that the pullback  bundle $f^\ast TM$ can be isometrically identified with the bundle $TN$, which induces an identification between $f^\ast S_M$ and $S_N$. Moreover, since both $\overbar F^t$ and $F^t$ are horospheres, we can choose the identification of $f^\ast TM$ and $TN$ so that the inner normal vector $\overbar e_n$ of $\overbar F^t$ is identified with the inner normal vector $e_n$ of $F^t$. We denote by $\partial_{\overbar r}\in TN$ the vector corresponding to $\partial_{r}\in TM$ under this identification. By the same argument in the proof of part $(iv)$, $\{e^{-\overbar r/2}\sigma_\lambda\}$ is a collection of parallel sections (with respect to the connection $\widetilde \nabla$ in line \eqref{eq:flatconn}) such that it forms an orthogonal basis of $S_N\otimes f^\ast S_M$ at every $x\in N$. Now the same argument for \cite[Theorem 2.9, part $(iii)$]{Wang2022:fl} finishes the proof. 	
\end{proof}

\begin{remark}
	Roughly speaking, part (III) of the above theorem asserts that $N$ and $M$ have the same geometric shape, although they could be of different sizes. 
\end{remark}
\begin{remark}
	We emphasize that there is \emph{not} much control over the global behavior of the map $f$ in Theorem \ref{thm:hyperbolicComparison-intro}. In particular,  in general the map $f$ does not preserve the leafwise structures, that is, $f$ generally does not map the horospheres of $N$ to the horospheres of $M$. See Theorem \ref{thm:hyperbolicPrismComparison} below for an example where the map $f$ can be rather arbitrary. 
\end{remark}
%
%

As an special case of Theorem \ref{thm:hyperbolicComparison-intro}, we have the following rigidity theorem for parabolic prisms in hyperbolic spaces.
\begin{definition}\label{def:prism}
	Let $P$ be a convex polyhedron in the Euclidean space $\R^{n-1}$. We view $\mathbb R^{n-1}$ as a leaf $\mathbb R^{n-1} \times \{r_0\}$ in $\mathbb H^n$, where $\mathbb H^n$ is viewed as a warped product space $\mathbb R^{n-1}\times \mathbb R$ with metric 
		$$g=dr^2+e^{2r}h.$$
A subspace $M=P\times I \subset \R^{n-1}\times \mathbb R = \mathbb H^n$ with the inherited metric is called  a parabolic  prism in $\mathbb H^n$.  
\end{definition}
We call the two components of $P\times\partial I$ the \emph{top} and \emph{bottom} of the parabolic prism $P\times I$. In particular, the top and bottom are flat and has mean curvature $n-1$ and $-(n-1)$ respectively. The side faces of $M$ (i.e., the codimension one face that lie in $\partial P\times I$)  are totally geodesic and have mean curvature zero.

\begin{theorem}[Theorem \ref{thm:hyperbolicPrismComparisonIntro}]\label{thm:hyperbolicPrismComparison}
	Let $(M, g)$ be a parabolic prism  in $\mathbb H^n$ and $f\colon (N, \overbar g)\to (M, g)$ a spin polytope map between $n$-dimensional manifolds with polyhedral boundary. 
	Assume that the degree of $f$ is non-zero. If
	\begin{enumerate}[label=$(\arabic*)$]
		\item 
		$\Sc(\overbar{g})_x \geq \Sc(g)_{f(x)} = -n(n-1)$ for all $x\in N$, 
		\item $H_{\overbar{g}}(\overbar{F}_i)_y \geq  H_{g}(F_i)_{f(y)}$ for all $y$ in each codimension one face $\overbar{F}_i$ of $N$, 
		\item $\theta_{ij}(\overbar{g})_z\leq \theta_{ij}(g)_{f(z)}$ for all $\overbar F_i, \overbar F_j$ and all $z \in \overbar F_i\cap \overbar{F}_j$,   
	\end{enumerate} 
	then \begin{enumerate}[label=$(\roman*)$]	
		\item 
		$\Sc(\overbar{g})_x = \Sc(g)_{f(x)}$ for all $x\in N$, 
		\item $H_{\overbar{g}}(\overbar{F}_i)_y =  H_{g}(F_i)_{f(y)}$ for all $y$ in each codimension one face $\overbar{F}_i$ of $N$, 
		\item $\theta_{ij}(\overbar{g})_z =  \theta_{ij}(g)_{f(z)}$ for all $\overbar F_i, \overbar F_j$ and all $z \in \overbar F_i\cap \overbar{F}_j$. 
	\end{enumerate} Furthermore, $N$ is also a parabolic prism in $\mathbb H^n$, and  is locally isometric to  $M$.
\end{theorem}
\begin{proof}
	By definition,  $M$ being a parabolic prism in $\mathbb H^n$ implies that $M$ is radially convex in $\mathbb H^n$ in the sense of Definition \ref{def:radiallyConvex}. Let us again write the metric $g$ on $M$ in the following warped product metric: 
	\[  g = dr^2 + e^{2r} h \]
	where $h$ is the flat metric on $\mathbb R^{n-1}$. Note that,  for a parabolic  prism, the normal vector of each codimension one face is either parallel or orthogonal to $\partial_r$. Since
		$$A=\langle\nu,\nabla r\rangle$$
		on any face of a parabolic prism, we do \emph{not} need to assume $f$ to be distance non-increasing on the boundary in the current setup, as indicated in the proof of Lemma \ref{lemma:secondff>=}. Therefore the proof of Theorem \ref{thm:hyperbolicComparison-intro} directly applies to give equalities of scalar curvature, mean curvature, and dihedral angles. Furthermore, by  Theorem \ref{thm:hyperbolicComparison-intro}, $(N, \overbar g)$ is also hyperbolic such that its top face is a horosphere. In fact, the same argument in the proof for part $(iv)$ of  Theorem \ref{thm:hyperbolicComparison-intro} also shows that the bottom face is a horosphere as well. Moreover, the normal vectors of each side face of $N$ is constant and parallel to the top and bottom due to part (III) of Theorem \ref{thm:hyperbolicComparison-intro}. Therefore $N$ is a parabolic  prism in $\mathbb H^n$. By \cite[Theorem 2.9]{Wang2022:fl}, the top face of $N$ is a convex polyhedron that is locally isometric to the top face of $M$. Since $N$ and $M$ are parabolic prisms in $\mathbb H^n$, it follows that $N$ is locally isometric to $M$. This finishes the proof. 
\end{proof}

As a consequence of Theorem \ref{thm:hyperbolicPrismComparison}, we have the following theorem, which is an analogue of Stoker's conjecture \cite{MR222765} for parabolic prisms in hyperbolic spaces. 

\begin{theorem}\label{thm:stokeranalogue}
	Suppose $M_1$ and $M_2$ are two parabolic prisms in $\mathbb H^n$ of the same combinatorial type such that all corresponding dihedral angles are equal, then all corresponding angles of $M_1$ and $M_2$ are equal. 
\end{theorem}
 We emphasize that the above theorem asserts that all corresponding angles (including the face angles\footnote{Face angles of any given face are dihedral angles of that face, when the face is viewed as a polyhedron itself. }  and the angles between \emph{non-adjacent} codimension one faces) of $M_1$ and $M_2$ coincide.

\section{Dihedral rigidity of direct product spaces}\label{sec:directprod}
In this section, we will prove Theorem \ref{thm:rigidityOfProduct-intro} for certain direct product spaces. We shall mainly focus on the case of direct products between flat spaces and  positively curved spaces. A simplest such example is the direct product $X\times I$ of a positively curved space $X$ and the unit interval $I$. This of course can also be viewed as a special case of warped product spaces, where the metric $g$ on $X\times I$ in this case is simply 
\[ g = dr^2 + h \]
for some fixed metric $h$ on $X$. In other words, the warping function $\varphi$ is the constant function  $\varphi \equiv 1$ in this case. 

Now let us prove Theorem \ref{thm:rigidityOfProduct-intro}. 

\begin{proof}[Proof of Theorem \ref{thm:rigidityOfProduct-intro}]
	By considering the direct product spaces $N\times I$ and $M\times I$ where $I$ is the unit interval, we can reduced the odd dimensional case to the even dimensional case. Hence without loss of generality, let us assume $n= \dim M$ is even. 

	It is an immediate consequence of  \cite[Theorem 1.7]{Wang:2021tq} that  the inequalities in (1), (2) and (3) have to be equalities, and moreover there exists a non-zero parallel section $\sigma$ of $S_N\otimes f^\ast S_M$ satisfying the boundary condition $B$. 
	
	We first prove that $TN$ admits an orthogonal decomposition that is  compatible with the decomposition of $TM = TX\oplus TY$. More precisely, we have the following. 
	\begin{claim*}
		There exists an orthogonal decomposition 
		\begin{equation}\label{eq:TN=U+V}
			TN=V_1\oplus V_2
		\end{equation} such that $f$ maps $V_1$ isometrically to $TX\subset TM$, and maps $V_2$ to $(TX)^\perp$.
	\end{claim*}
	Let $p_X$ be the projection from $M$ to $X$, and $P_X$ be the orthogonal projection from $TM$ to $TX$. For each $x\in N$, consider the map $P_Xf_*\colon T_xN\to T_{f(x)}M$. By diagonalizing  $P_Xf_*$ with respect to the inner products on $T_xN$ and $T_{f(x)}M$, there exist $\mu_j\geq 0$ and  orthonormal bases $\{\overbar e_j\}_{1\leq j\leq n}$ for $T_x N$ and $\{e_j\}_{1\leq j\leq n}$ for $T_{f(x)}M$ such that $P_Xf_*(\overbar e_j)=\mu_j e_j$. 
	
	Let $S$ be the subset of $\{1,\ldots ,n\}$ such that $\mu_j\ne 0$. Note that only the $TX$-components of the curvature form of $M$ is non-zero. Therefore, by \cite[Lemma 2.1]{Wang:2021tq} or similarly the proof of Lemma \ref{lemma:curvature>=},  the equality of scalar curvature $\Sc(\overbar{g})_x = \Sc(g)_{f(x)}$  implies that
	$$\sum_{i,j\in S}(1-\mu_i\mu_j) R_{ijij}^X=0$$
	where $R_{ijij}^X$ the curvature form of $g_X$. It follows from the positivity of the Ricci curvature of $g_X$ that $\mu_j=1$ for $j\in S$. Since $f$ is distance non-increasing, the tangent map $f_\ast$ maps  $\ker(P_Xf_*)^\perp$ isometrically  to $TX$. Define $V_1$ to be  $\ker(P_Xf_*)^\perp$. Then  its orthogonal complement $V_2 = V_1^\perp$ is mapped to $(TX)^\perp$, which is the flat part of $TM$. This finishes the proof of the claim.

	Let $F^Y_i$ be a codimension one face of $Y$. Then $F_i \coloneqq X\times F^Y_i$ is a codimension one face of $X\times Y$. Suppose $\overbar F_i$ is a codimension one face of $N$ that maps to $F_i$.  Let $x$ be a point in the interior of $\overbar F_i$  such that $f(x)$ also lies in the interior of $F_i$. 
	Let $\overbar \nu_x$ (resp. $\nu_x$) be the unit inner normal vector at $x$ (resp. $f(x)$). Recall that there exist a nonzero parallel section $\sigma$ of $S_N\otimes f^\ast S_M$ over $N$ such that $\sigma$ satisfies the boundary condition
	$$\mathscr E(\overbar c(\overbar \nu_x)\otimes c(\nu_x))\sigma=-\sigma,$$
	where $\mathscr E = \overbar \epsilon \otimes \epsilon$ with $\overbar \epsilon$ and $\epsilon$ the grading operator on $S_N$ and $f^\ast S_M$. Equivalently, we have $\overbar c(\overbar \nu_x)\sigma=-\mathscr E c(\nu_x)\sigma$ at $x$.

	Since $\nu_{x}$ lies in the flat part $TY$ of $TM$, it extends to a parallel vector field on $M$, which will be denoted by $\nu$. Observe that  $-\mathscr Ec(\nu)\sigma$ on $N$, in particular, is parallel along any piecewise smooth path in $N$.  We shall show that $\overbar \nu_x$ also extends to a parallel vector field on $N$. Let $\gamma_t,~t\in[a,b]$,  be any piecewise smooth loop in $N$ such that $\gamma_a=\gamma_b=x$. Let $\overbar \nu_{\gamma_t}$ be the parallel transport of $\overbar \nu_x$ from $\gamma_a$ to $\gamma_b$ along $\gamma$. Therefore $\overbar c(\overbar \nu_{\gamma_t})\sigma(\gamma_t)$ is parallel along $\gamma$. Observe that  $\overbar c(\overbar \nu_{\gamma_a})\sigma(x)=-\mathscr E c(\nu_x)\sigma(x)$ at $\gamma_a=x$. Since $-\mathscr Ec(\nu)\sigma$ is also parallel along $\gamma$, it follows from the uniqueness theorem of first order differential equations that 
	$$\overbar c(\overbar \nu_{\gamma_b})\sigma(\gamma_b)= -\mathscr E c(\nu_{\gamma_b})\sigma(\gamma_b) =  -\mathscr E c(\nu_{x})\sigma(x) =  \overbar c(\overbar \nu_{\gamma_a})\sigma(x),$$
	hence $\overbar \nu_{\gamma_b}=\overbar \nu_{x}$. It follows that the parallel transport of $\overbar \nu_x$ from $x$ to any point $x'\in N$ is independent of the choice of path. In particular,  there is a unique way to extend $\overbar \nu_x$ to parallel vector field over $N$, which will be denoted by $\overbar \nu$. 
	
	If $F^Y_{i_1}, \cdots, F^Y_{i_\ell}$ are $\ell$ codimension one faces of $Y$ and $x$ lies in the interior of the intersection $\bigcap (X\times F^Y_{i_j})$, then a similar argument also generates $\ell$ linearly independent parallel vector fields along $N$ (cf. the proof of  \cite[Theorem 2.9]{Wang2022:fl}). To summarize, since $TY$ admits $k$ linearly independent orthonormal parallel sections $v_1, \cdots v_k$, the above process produces $k$ parallel sections $\overbar v_1, \cdots, \overbar v_k$ of $TN$ such that 
	$$\overbar c(\overbar v_i)\sigma=-\mathscr E c(v_i)\sigma$$
	everywhere on $N$. Since $\sigma$ is nowhere vanishing, it follows that 
	\[ |\overbar v_i|_{\overbar g} = |v_i|_g \]
	everywhere on $N$ for all $i$. By the polarization identity, we have 
	$$\langle \overbar v_i,\overbar v_j\rangle=\langle v_i, v_j\rangle$$
	everywhere on $N$. In particular, $\{\overbar v_1, \cdots, \overbar v_k\}$ is a set of $k$ linearly independent orthonormal parallel section of $TN$,  which naturally defines $k$ coordinate functions $\rho_1,\ldots,\rho_k$ on $N$ (with respect to some given point in $N$, where this given point is viewed as the origin of $\mathbb R^k$). 
	
 Since $\sigma$ is parallel, the curvature form of $\nabla$ acts as the zero operator on  $\sigma$, that is, 
\[ (R^{S_N}_{\overbar e_a, \overbar e_b}  + R^{S_M}_{f_*\overbar e_a,f_*\overbar e_b}) \sigma = 0 \]
for all $\overbar e_a, \overbar e_b\in N$,  where 
\[  R^{S_N}_{\overbar e_a, \overbar e_b} \coloneqq  \frac{1}{2} \sum_{i<j} \langle R^N_{\overbar e_a, \overbar e_b} \overbar e_i, \overbar e_j\rangle \overbar c(\overbar e_i) \overbar c(\overbar e_j)  \]
and  
\[  R^{S_M}_{f_\ast \overbar e_a, f_\ast \overbar e_b}\coloneqq  \frac{1}{2} \sum_{i<j} \langle R^M_{f_\ast \overbar e_a, f_\ast \overbar e_b}  e_i, e_j\rangle c(e_i) \overbar c(e_j)  \]
with $\{ \overbar e_i\}$ (resp. $\{e_i\}$)  an orthonormal basis of $TN$ (resp. $TM$).  Since $\overbar v_1, \cdots, \overbar v_k$ are parallel, the curvature form $R^N_{\overbar e_a, \overbar e_b}$ of $TN$ vanishes if either $\overbar e_a$ or $\overbar e_b$ lies in the linear span of   $\overbar v_1, \cdots, \overbar v_k$. It follows that  for each $\overbar v_i$ from above, we have 
	$$ R^{S_M}_{f_*\overbar v_i,f_*\overbar w} (\sigma)=0,$$
	 for any $\overbar w \in TN$. From line \eqref{eq:TN=U+V}, the subbundle $TX$ of $TM$ is contained in the image $f_*(TN)$. Together with the fact the orthogonal complement of $TX$ in $TM$ is flat,  this implies
	 \[   R^{S_M}_{f_*\overbar v_i, w}(\sigma) = 0 \] for all $w\in TM$. It follows from \cite[Corollary 2.8]{spinorialapproach} that the Ricci curvature operator of $TM$ vanishes at $f_*\overbar v_i$, hence $f_*\overbar v_i$ lies in the flat part $TY$ of $TM$. Thus the tangent bundle of submanifold $\rho^{-1}(y)$ is  equal to $V_1$ from line \eqref{eq:TN=U+V}, where $\rho = (\rho_1, \cdots, \rho_k)\colon N \to \mathbb R^k$ is the coordinate function determined by the parallel vector fields $\overbar v_1, \cdots, \overbar v_k$ from above and $y\in \mathbb R^k$ is some fixed element of $\mathbb R^k$.  In particular, the map $f_\ast$ maps the tangent bundle of $\rho^{-1}(y)$  isometrically to $TX$. This implies that $f\colon \rho^{-1}(y)\to X$ is a local isometry of manifolds with polyhedral boundary. It follows that $f\colon \rho^{-1}(y)\to X$ is a Riemannian covering map, since $\rho^{-1}(y)$ is compact and $X$ is connected. Therefore $N$ is isometric to a direct product of manifolds with polyhedral boundary $(\overbar X\times \overbar Y, g_{\overbar X}\times g_{\overbar Y})$ such that $\overbar Y$ is a flat manifold with polyhedral boundary and 
	 \[ f = (f_X, f_Y) \colon \overbar X \times \overbar Y \to X\times Y \]
with  $f_X\colon (\overbar  X, g_{\overbar X}) \to (X, g)$  a local isometry. 

By assumption,  we have  	 $H_{\overbar{g}}(\overbar{F}_i)_y \geq  H_g(F_i)_{f(y)}$ for all $y$ in each codimension one face $\overbar{F}_i$ of $N$, and  
$\theta_{ij}(\overbar{g})_z\leq \theta_{ij}(g)_{f(z)}$ for all $\overbar F_i, \overbar F_j$ and all $z \in \overbar F_i\cap \overbar{F}_j$. Hence the same comparison inequalities on mean curvature and dihedral angles  hold between $\overbar Y$ and $Y$. It follows from \cite[Theorem 2.9]{Wang2022:fl} that  $\overbar Y$ 
  	also a convex subspace in $\R^k$ and  there exists an orthogonal transformation $\Theta$ of $\mathbb R^k$ such that, for all $y$ in the interior of codimension one faces of $\overbar Y$, $\Theta$ maps the unit inner normal vector at  $y\in \partial \overbar Y$ to the unit inner normal vector at $f_Y(y)$, provided $f_Y(y)$ also lies in the interior of codimension one faces of $Y$.
 This finishes the proof. 
\end{proof}

\end{document}